\documentclass{amsart}

\usepackage[english]{babel}
\usepackage[latin1]{inputenc}
\usepackage{amsmath, amsthm, amssymb, stmaryrd, xcolor, enumitem}
\usepackage{bm}

\definecolor{linkred}{rgb}{0.6,0,0}
\definecolor{linkblue}{rgb}{0,0,0.6}
\usepackage[
	pdftitle={The Moduli Space of Stables Maps with Divisible Ramification},
	pdfauthor={Oliver Leigh},
	ocgcolorlinks,
	linkcolor=linkblue,
	citecolor=linkred,
	urlcolor=linkred]
{hyperref}

\usepackage[all]{xy} 

\usepackage{tikz}
\usepackage{tikz-cd}
\usetikzlibrary{%
  matrix,%
  calc,%
  arrows%
}

\newcounter{mainTheoremCounter}
\newtheoremstyle{mainTheorem}{}{}{\addtolength{\leftskip}{0em}\itshape}{}{\bfseries}{}{.5em}{\thmname{#1} {#2}}
\theoremstyle{mainTheorem}
\newtheorem{mainTheorem}[mainTheoremCounter]{Theorem}

\newcounter{mainDefinitionCounter}
\newtheoremstyle{mainDefinition}{}{}{}{}{\bfseries}{}{.5em}{\thmname{#1} {#2}}
\theoremstyle{mainDefinition}
\newtheorem{mainDefinition}[mainDefinitionCounter]{Definition}

\newcounter{mainRemarkCounter}
\newtheoremstyle{mainRemark}{}{}{}{}{\bfseries}{}{.5em}{\thmname{#1} {#2}}
\theoremstyle{mainRemark}
\newtheorem{mainRemark}[mainRemarkCounter]{Remark}

\newcounter{appendixTheoremCounter}
\newtheoremstyle{appendixTheorem}{}{}{\addtolength{\leftskip}{0em}\itshape}{}{\bfseries}{}{.5em}{\thmname{#1} {#2}}
\theoremstyle{appendixTheorem}

\theoremstyle{plain}  

\newtheorem{theorem}{Theorem}[subsection]
\newtheorem{lemma}[theorem]{Lemma}

\newtheorem{corollary}[theorem]{Corollary}

\theoremstyle{definition} 

\newtheorem{definition}[theorem]{Definition}

\newtheorem{re}[theorem]{}
\newtheorem{remark}[theorem]{Remark}

\newcommand{\SmNeg}{\mbox{\scriptsize-}}

\newcommand{\C}{\mathbb{C}}

\newcommand{\bbL}{\mathbb{bbL}}

\renewcommand{\P}{\mathbb{P}}

\newcommand{\calC}{\mathcal{C}}

\newcommand{\calE}{\mathcal{E}}
\newcommand{\calF}{\mathcal{F}}
\newcommand{\calG}{\mathcal{G}}
\newcommand{\calH}{\mathcal{H}}
\newcommand{\calR}{\mathcal{R}}
\newcommand{\calT}{\mathcal{T}}
\newcommand{\calX}{\mathcal{X}}
\newcommand{\calY}{\mathcal{Y}}

\newcommand{\frakM}{\mathfrak{M}}

\renewcommand{\O}{\mathcal{O}}
\newcommand{\Spec}{\mathrm{Spec}}

\newcommand{\Sym}{\mathrm{Sym}}
\newcommand{\Supp}{\mathrm{Supp}}
\renewcommand{\div}{\mathrm{div}}

\newcommand{\M}{\mathcal{M}}
\newcommand{\Mbar}{\overline{\mathcal{M}}}
\renewcommand{\L}{\mathcal{L}}

\newcommand{\smallfamily}[3]{\begin{array}{c}\xymatrix@=1em{{#1}\ar[d]^{#3} \\ {#2}}\end{array}}

\begin{document}
\setlength{\parindent}{0cm}

\title{The Moduli Space of Stable Maps with Divisible Ramification}

\author{Oliver Leigh}
\address{Oliver Leigh, School of Mathematics and Statistics, The University of Melbourne, Victoria, 3010, Australia.}
\address{Oliver Leigh, Department of Mathematics, The University of British Columbia, Vancouver, BC, V6T 1Z2, Canada.}
\curraddr{}
\email{oleigh@math.ubc.ca}
\thanks{}


\date{}
\dedicatory{}

\begin{abstract}
We develop a theory for stable maps to curves with divisible ramification. For a fixed integer $r>0$, we show that the condition of every ramification locus being divisible by $r$ is equivalent to the existence of an $r$th root of a canonical section. We consider this condition in regards to both absolute and relative stable maps and construct natural moduli spaces in these situations. We construct an analogue of the Fantechi-Pandharipande branch morphism and when the domain curves are genus zero we construct a virtual fundamental class. This theory is anticipated to have applications to $r$-spin Hurwitz theory. In particular it is expected to provide a proof of the $r$-spin ELSV formula [\hyperlink{SSZ_abstract}{SSZ'15, Conj. 1.4}] when used with virtual localisation.  
\end{abstract}

\maketitle

\section*{Introduction}

Consider a smooth curve $X$ and the moduli space parameterising degree $d$ maps $f:C\rightarrow X$  where $C$ is a smooth curve of genus $g$. This space is denoted by $\M_{g}(X,d)$ and point $[f]\in \M_{g}(X,d)$ has an associated exact sequence
\begin{align}
0\longrightarrow  f^*\Omega_{X}\otimes \Omega_{C}^\vee \overset{\delta^\vee}{\longrightarrow} \O_{C} \longrightarrow \O_{R_f} \longrightarrow 0 \label{ram_exact_seq}
\end{align}
where $R_f$ is the ramification divisor. If $[f]$ is a generic point then $R_f$ is the union of disjoint points on $C$. In other words, $f$ has simple ramification everywhere. \\

As an alternatively, we consider a space $\M^{_{1/r}}_g(X, d)$ where a generic point $[f]$ gives a ramification divisor of the form  $R_f = r\cdot p_1 + \cdots + r \cdot p_m$ for disjoint points $p_1, \ldots ,p_m\in C$. Specifically, we define $\M^{_{1/r}}_g(X, d)$ as the  following sub-moduli space of $\M_{g}(X,d)$:
\[
\M^{_{1/r}}_g(X, d) = \Big\{~ \big[f:C\rightarrow X \big] \in \M_g(X, d)~\Big|~  \mbox{$R_f = r\cdot D$ for some $D\in\mathrm{Div}(C)$} ~\Big\} /\sim .
\]

In this article we construct a natural compactification of $\M^{_{1/r}}_g(X, d)$. We develop the enumerative geometry of this space by constructing a virtual fundamental class in the case $g=0$ and by constructing a branch morphism. \\

The above construction of the ramification divisor relies on the domain curve $C$ being smooth. This means that $\Omega_{C}$ is locally free and that $df:f^*\Omega_X\rightarrow \Omega_C$ is injective.  If $C$ is allowed to be singular either of these may be false and we no longer have a straightforward definition of ramification. This leads us to rephrase the moduli problem using $r$th roots of $\delta$ which is defined in (\ref{ram_exact_seq}).  One can show that $\M^{_{1/r}}_g(X, d)$ is naturally isomorphic to:
\[
\left\{  \big[f\hspace{-0.1em}:\hspace{-0.1em}C\hspace{-0.1em}\rightarrow\hspace{-0.15em} X \big] \in \M_g(X, d)  \left| 
\begin{array}{l} 
\mbox{There is a line bundle $L$ on $C$, $\sigma \in H^0(L)$ and} \\ 
\mbox{an isom. $L^{\otimes r}\overset{e}{\rightarrow}\omega_C\otimes f^*\omega_X$ with $e(\sigma^r) = \delta$.}
\end{array}\hspace{-0.3em}\right\}\right. /\sim .
\]

We now have the moduli problem in a form which can be naturally compactified. First we note that for nodal domain curves there is a natural morphism $\Omega_C \rightarrow \omega_C$. This is combined with the differential map  $df: f^*\Omega_{X} \rightarrow \Omega_C$ to obtain a morphism which we denote by
\begin{align}
\delta: \O_C \longrightarrow \omega_C \otimes f^* \omega^\vee_{X}.  \label{delta_defn_1}
\end{align}\\\vspace{-0.5cm}

\begin{mainDefinition}\label{mainModuliDefinition} Denote by $\Mbar^{_{1/r}}_g(X, d)$ the moduli stack parameterising morphisms $f: C\rightarrow X$ where
				\begin{enumerate}
					\item $C$ is a genus $g$ $r$-prestable curve (a stack such that the coarse space $\overline{C}$ is a prestable curve, where points mapping nodes of $\overline{C}$ are balanced $r$-orbifold points, and ${C}^{\mathrm{sm}} \cong \overline{C}^{\mathrm{sm}}$);
					\item $f$ is a morphism such that the induced morphism $\overline{f}:\overline{C}\rightarrow X$ on the course space is a stable map;
					\item there exists a line bundle $L$ on $C$, an isomorphism $ e: L^{\otimes r} \overset{\sim}{\rightarrow} \omega_C \otimes f^* \omega^\vee_{X}$, and a morphism $\sigma: \O_C \rightarrow L$ such that $e(\sigma^r) = \delta$, where $\delta$ is defined in (\ref{delta_defn_1}).\\
				\end{enumerate}
					
\end{mainDefinition}

\begin{mainRemark}\label{main_relative_remark}
Throughout the paper we will also be considering the same moduli problem in the context of stable maps relative to a point $x\in X$ and a partition $\mu$ of $d>0$. The moduli space of relative stable maps $\Mbar_g(X,\mu)$ generically parameterises maps where the pre-image of $x$ is smooth and locally has monodromy given by $\mu$. We will leave the specifics of this moduli problem until section \ref{stable_and_relative_overview_section}, however all of the following results will hold when $\Mbar^{_{1/r}}_g(X, d)$ is replaced by $\Mbar^{_{1/r}}_g(X, \mu)$, and $2g-2-d(2g_X-2)$ is replaced by $2g-2+l(\mu)+|\mu|(1-2g_X)$.\\
\end{mainRemark}

\begin{mainRemark}
The $r$-prestable curves in definition \ref{mainModuliDefinition} arise naturally when taking $r$th roots of line bundles on nodal curves \cite{AbraJarvis,Chiodo_StableTwisted}. We review this in section \ref{r-prestable_curves_section}.  \\
\end{mainRemark}

\begin{mainTheorem}\label{construction_and_proper_theorem}
$\Mbar^{_{1/r}}_g(X, d)$ is a proper DM stack. It is non-empty only when $r$ divides $2g-2-d(2g_X-2)$. The natural forgetful map
\[
\chi:\Mbar^{_{1/r}}_g(X, d) \longrightarrow \Mbar^{}_g(X, d)
\]
is both flat and of relative dimension $0$ onto its image. It is an immersion when restricted to $\M^{_{1/r}}_g(X, d)$.\\
\end{mainTheorem}

The image of $\chi$ has an explicit point-theoretic description.  Let $f:C\rightarrow X$ be a stable map and consider the locus in $C$ where $f$ is not \'{e}tale. Following  \cite{Vakil_Enum, GraVakil} a connected component of this locus is called a \textit{special locus}. A special locus will be one of the following: 
\begin{enumerate}[label=(\alph*)]
\item A smooth point of $C$ where $f$ is locally of the form $z\mapsto z^{a+1}$ with $a\in\mathbb{N}$.
\item A node of $C$ such that on each branch $f$ is locally like $z\mapsto z^{a_i}$ with $a_i\in\mathbb{N}$.
\item A genus $g'$ component $B$ of $C$ where $f|_{B}$ is constant and on the branches of $C$ meeting $B$ the map  $f$ is locally of the form $z\mapsto z^{a_i}$ with $a_i\in\mathbb{N}$.
\end{enumerate}

Note that a slightly different definition is used for the relative case (see remark \ref{relative_special_loci}). Now, following  \cite{Vakil_Enum, GraVakil} again, we define a \textit{ramification order} (or sometimes simply \textit{order}) for each type of special locus by:
\begin{enumerate}[label=(\alph*)]
\item $a$.  \hspace{1.25cm} (b) $a_1+a_2$. \hspace{1.25cm} (c) $2g'-2 + \sum(a_i+1)$.
\end{enumerate}
This gives us an extended concept of ramification. There is also an extended concept of branching constructed in \cite{FantechiPand} which agrees with the ramification order assigned to special loci. Specifically there is a well defined morphism of stacks which agrees with the classical definition of branching on the smooth locus: 
\[
\begin{array}{cccc}
br : &\M_g(X,d) &\longrightarrow &\Sym^{2g-2-d(2g_X-2)} X. 
\end{array}
\]

\begin{figure}
  \centering
      \includegraphics[width=0.7\textwidth]{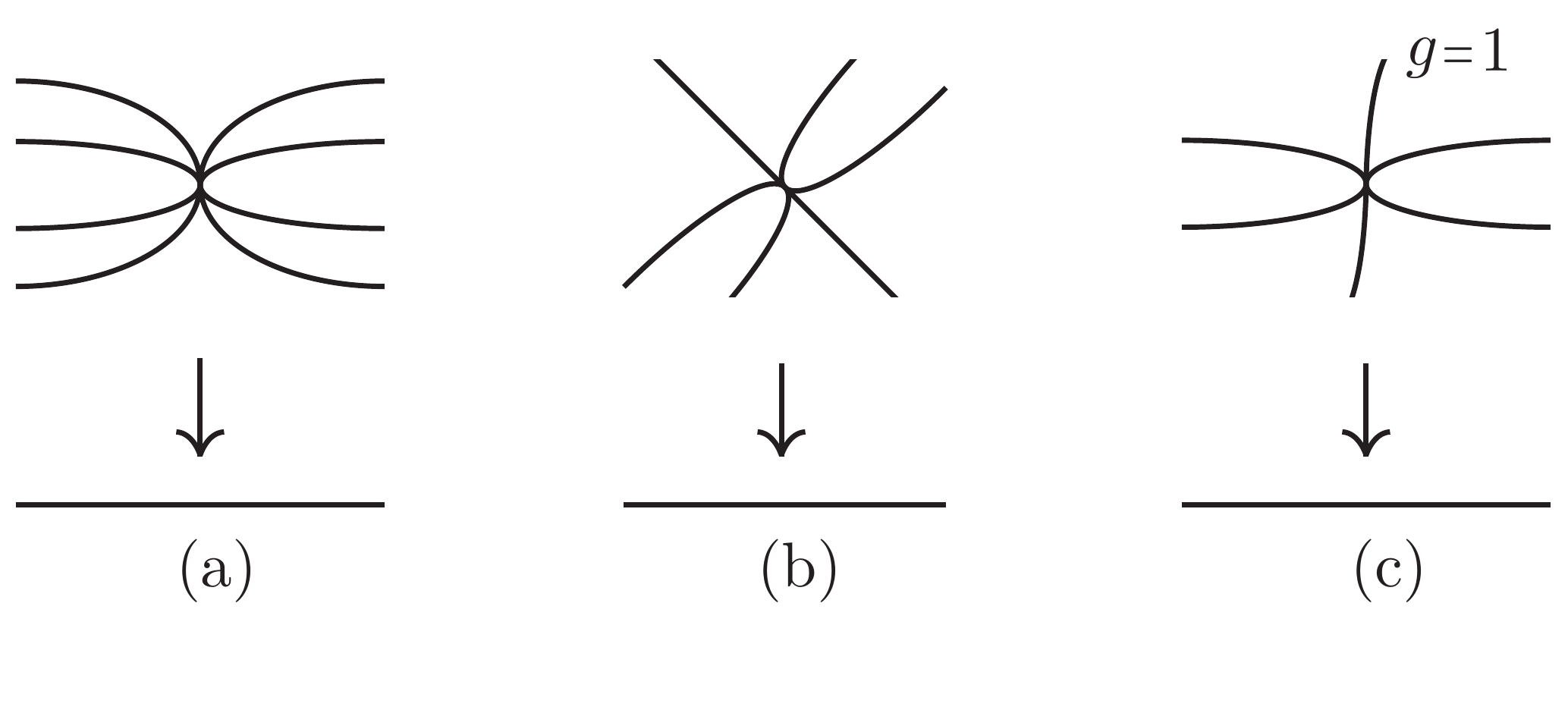}\vspace{-1.25em}
  \caption{Loci with ramification order 3. (a) A smooth point where the map is locally like $z\mapsto z^{3+1}$. (b) A node where the map is locally like $z\mapsto z^{2}$ on one branch and $z\mapsto z$ on the other. (c) A genus one component meeting its complement at a node, where the map is constant on the sub-curve and locally like $z\mapsto z^{2}$ on the complement. }
\end{figure}

\begin{mainTheorem}\label{br_theorem}
The objects is $\Mbar^{_{1/r}}_g(X, d)$ have the following ramification and branching properties:
\begin{enumerate}
\item\label{br_theorem_moduli_points}
The closed points in the image of $\tau:\Mbar^{_{1/r}}_g(X, d) \longrightarrow \Mbar^{}_g(X, d)$ are the closed points of  $\Mbar^{}_g(X, d)$ with the property: 
\begin{itemize}
\item[]\hspace{-1.5em} ``Every special locus of the associated map has order divisible by $r$''.\\\vspace{-0.25cm}
\end{itemize}
\item \label{br_theorem_branch_morphism}
There is a morphism of stacks
\[
\mathtt{br} : \Mbar^{_{1/r}}_g(X,d) \longrightarrow \Sym^{\frac{1}{r}(2g-2-d(2g_X-2))} X 
\]
that commutes with the branch morphism of \cite{FantechiPand} via the diagram
\[
\xymatrix@R=1.5em{
 \Mbar^{_{1/r}}_g(X,d) \ar[r]^{\mathtt{br}\hspace{2.1em}}\ar[d]_{\chi} &\Sym^{\frac{1}{r}(2g-2-d(2g_X-2))} X \ar[d]^{\Delta}\\
 \Mbar_g(X,d) \ar[r]^{br\hspace{2.1em}} & \Sym^{2g-2-d(2g_X-2)} X  
}
\]
where $\Delta$ is defined by $\sum_i x_i \mapsto \sum_i rx_i$.
\\\vspace{-0.25cm}
\end{enumerate}
\end{mainTheorem}

Just like for regular stable maps $\Mbar_g(X, d)$, the smooth-domain locus $\M^{_{1/r}}_g(X, d)$ can be empty while $\Mbar^{_{1/r}}_g(X, d)$ is non-empty. For explicit examples consider degree one maps to $\P^1$ with $g>0$. \\

The properties of $\Mbar^{_{1/r}}_g(X, d)$ can be quite different to those of $\Mbar_g(X, d)$. For example, if we consider genus zero domains we have that $\Mbar_0(X, d)$ is smooth, but in general $\Mbar^{_{1/r}}_0(X, d)$ is not. An explicit example of this is $\Mbar^{_{1/3}}_0(\P^1, 4)$, which is not smooth  as it contains components of dimensions $2$ and $3$. However, we do have the existence of a virtual fundamental class for $g=0$. \\

\begin{mainTheorem}\label{pot_theorem}
$\Mbar^{_{1/r}}_0(\P^1, d)$ has a natural perfect obstruction theory giving a virtual fundamental class of dimension $\frac{1}{r}(2d-2) =\frac{1}{r}\mathrm{virdim}(\Mbar_0(\P^1, d))$. \\
\end{mainTheorem}

The moduli space $\Mbar^{_{1/r}}_g(\P^1, \mu)$ has expected applications to $r$-spin Hurwitz theory. For example, in genus $0$ using both theorems \ref{br_theorem} and \ref{pot_theorem} we have the following natural intersection 
\begin{align}
\int_{[\Mbar^{_{1/r}}_0(\P^1, \mu)]^{\mathrm{vir}}} \mathtt{br}^* H^{\frac{1}{r}(l(\mu)+|\mu|-2)} \label{hurwitz_integral}
\end{align}
where $H$ is the hyperplane class in $\Sym^{\frac{1}{r}(l(\mu)+|\mu|-2)} \P^1 \cong \P^{\frac{1}{r}(l(\mu)+|\mu|-2)}$. This is a direct analogue of the characterisation of simple Hurwitz numbers given in \cite[Prop. 2]{FantechiPand}. This was the first step towards a proof via virtual localisation of the ELSV formula. After applying the virtual localisation techniques of \cite{GraberPand}, (\ref{hurwitz_integral}) is expected to be related to the $r$-ELSV formula of \cite{SSZ_rELSV,BKLPS}. \\

In the case where $r = l(\mu) +|\mu|-2$ the space $\Mbar^{_{1/r}}_0(\P^1, \mu)$ has virtual dimension $1$. These spaces are characterised by having exactly one free special locus of order $r$. In this situation the intersections given in (\ref{hurwitz_integral}) are expected to have a direct relation to the completion coefficients and one-point invariants of \cite{OkounkovPand_Comp}:
\[
\int_{[\Mbar_{0,1}(\P^1, \mu)]^{\mathrm{vir}}} \psi_1^r ev_1^*[pt].
\]\\

The paper is structured as follows: 
\begin{itemize}[leftmargin=0cm]
\item[] \textbf{Section \ref{background_section}:} Review the necessary theory of stable maps, $r$-prestable curves and line bundles on twisted curves required for the construction of $\Mbar^{_{1/r}}_g(X, d)$ and its relative version.
\item[] \textbf{Section \ref{construction_section}:} Extend the theory of roots of line bundles to the space of stable maps, construct the moduli space $\Mbar^{_{1/r}}_g(X, d)$ and  then prove theorem \ref{construction_and_proper_theorem}.
\item[] \textbf{Section \ref{br_section}:} Consider properties of $\Mbar^{_{1/r}}_g(X, d)$ related to branching and ramification while proving theorem \ref{br_theorem}. 
\item[] \textbf{Section \ref{pot_section}:} Consider the cotangent complex of $\Mbar^{_{1/r}}_g(X, d)$ and related properties while proving theorem \ref{pot_theorem}. \\
\end{itemize}

\textbf{Conventions} All stacks and schemes and schemes are over $\C$. By \textit{local picture} we will mean the following. Let $f:X\rightarrow Y$ and $g:U\rightarrow V$ be morphisms of stacks.  The \textit{local picture} of $f$ at $x\in X$  is the same as the local picture of $g$ at $u \in U$ if:
\begin{itemize}
\item There is an isomorphism between the strict henselization $f^{\mathrm{sh}}:X^{\mathrm{sh}}\rightarrow Y^{\mathrm{sh}}$ of $f$ at $x$ and the strict henselization $g^{\mathrm{sh}}:U^{\mathrm{sh}}\rightarrow V^{\mathrm{sh}}$ of $g$ at $u$.
\end{itemize}

Throughout the paper we will consider both absolute and relative stable maps. The theory will be similar so we introduce the following simplifying notation.\\

\hypertarget{notation_main}{\textbf{Notation:}} 
\begin{itemize}
\item  $\M$ is either $\Mbar_{g}(X,d)$ or $\Mbar_{g}(X,\mu)$ for $g \geq 0$, $d>0$ and $\mu$ a partition of $d$.
\item $\calC\rightarrow \M$ be the associated universal curve. 
\item  If $\M$ is $\Mbar_{g}(X,d)$ (resp. $\Mbar_{g}(X,\mu)$) then $\frakM$ is $\frakM_g$ (resp. $\frakM_{g,l(\mu)}$).
\item The expected number of order $r$ special loci in the generic case is denoted by $m$. When $\M=\Mbar_g(X,d)$  we have $m = \frac{1}{r} (2g-2-d(2g_X-2))$ and when $\M=\Mbar_g(X,\mu)$ we have $m = \frac{1}{r}  (2g-2+l(\mu)+|\mu|(1-2g_X))$.
\item Whenever we introduce a new related space, the notation will carry through to the original spaces. For $\Mbar_{g}(X,\mu)$ the two key spaces are:
\begin{itemize}
\item $\M^{[r]} = \Mbar^{\,[r]}_{g}(X,\mu)$ and $\calC^{[r]} = \overline{\calC}^{\,[r]}_{g}(X,\mu)$ defined in \ref{def_stable_maps_with_roots_of_the_ramification_bundle}.\\
\item $\M^{[\frac{1}{r}]} = \Mbar^{_{[1/r]}}_{g}(X,\mu)$ and $\calC^{[\frac{1}{r}]} = \overline{\calC}^{_{[1/r]}}_{g}(X,\mu)$ are defined in \ref{main_moduli_spaces_definitions}.\ref{full_space_part}.\\
\end{itemize}
And the associated spaces are:
\begin{itemize}
\item $\M^{r} = \Mbar^{\,r}_{g}(X,\mu)$ and $\calC^{r} = \overline{\calC}^{\,r}_{g}(X,\mu)$ defined in \ref{moduli_with_roots_of_ramification_section}.\\
\item $\M^{{\frac{1}{r}, \calE}} = \Mbar^{_{\frac{1}{r}, \calE}}_{g}(X,\mu)$ and $\calC^{{\frac{1}{r}, \calE}} = \overline{\calC}^{_{\frac{1}{r}, \calE}}_{g}(X,\mu)$ are defined in \ref{roots_of bundle_moduli_definition}.\\
\item $\M^{\frac{1}{r}} = \Mbar^{_{1/r}}_{g}(X,\mu)$ and $\calC^{1/r} = \overline{\calC}^{_{1/r}}_{g}(X,\mu)$ are defined in \ref{main_moduli_spaces_definitions}.\ref{reduced_space_part}.\\
\end{itemize}
\end{itemize}

\textbf{Acknowledgements:} I would like to thank Jim Bryan and Paul Norbury for their knowledge, wisdom and support. I would also like to thank Kai Behrend, Emily Clader, Barbara Fantechi, Felix Janda, Martijn Kool, Georg Oberdieck, J\o rgen Rennemo, and Dustin Ross for very helpful conversations. \\

Parts of this paper were completed during visits to the Max Planck Institute for Mathematics, Ludwig Maximilian University of Munich, the Bernoulli Center, and the Henri Poincaré Institute. I would like to thank them for providing stimulating and welcoming work environments. 

\section{Review of Stable Maps and \texorpdfstring{$r$}{r}-Stable Curves}
\label{background_section}

\subsection{Stable Maps and Relative Stable Maps}\label{stable_and_relative_overview_section}

For the rest of this article we will set $X$ to be a non-singular curve. Recall that a stable map $f : C \rightarrow X$ is a degree $d$ morphism from a genus $g$ prestable curve to $X$ which has no infinitesimal automorphisms. We denote by $\Mbar_g(X,d)$ the moduli stack of these objects. Specifically this is the groupoid containing the objects: 
\[
\xi = \big(~
\pi: C \rightarrow S
,~~f : C \rightarrow X
~\big)
\]
where $\pi$ is a proper flat morphism and for each geometric point $p\in S$ we have $f_p: C_p \rightarrow X$ is a degree $d$ genus $g$ stable map to $X$. A morphism $\xi_1\rightarrow \xi_2$  in $\Mbar_g(X,d)$ between objects $\xi_i = (\pi_i:C_i\rightarrow S_i, f_i:C_i\rightarrow X)$ is a commutative diagram where the left square is cartesian:
\[
\xymatrix@=1.5em{
 S_1\ar[d] &C_1 \ar[l]_{\pi_1}\ar[d]\ar[r]^{f_1}&\ar@{=}[d]X\\
S_2  &\ar[l]_{\pi_2} C_2\ar[r]^{f_2}&X
}
\]\\

Let $x$ be a geometric point of $X$ and $\mu$ a partition of $d>0$. As we mentioned in remark \ref{main_relative_remark}, we will also be considering the moduli problem in the case of stable maps relative to $(x,\mu)$. We use the algebro-geometric definition of this moduli space and its obstruction theory provided in \cite{li1,li2}.\\

The goal of relative stable maps is to parameterise maps where the pre-image of $x$ lies in the smooth locus of $C$ and where the map has monodromy given by $\mu$ locally above $x$. However, this condition will not give a compact space. The solution provided in \cite{li1} is to allow the target to degenerate in a controlled manner by allowing $X$ sprout a chain of $\P^1$'s.\\

Specifically we can define the $n$th degeneration $X[i]$ inductively from $X[0]:=X$ by:
\begin{itemize}
\item $X[i+1]$ is given by the union $X[i] \cup \P^1$ meeting at a node $n_{i+1}$.
\item The node $n_{1}$ is at $x \in X$. For $i>0$ the node $n_{i+1}$ is in the $i$th component of $X[i+1]$, i.e. the node is not in $X[i-1] \subset X[i+1]$.
\end{itemize}
Then a \textit{degenerated target} is a pair $(T, t)$ where $T=X[i]$ for some $i\geq 0$ and $t$ is a geometric point in the smooth locus of $i$th component of $T$.  \\

A genus $g$ stable map to $X$ relative to $(\mu, x)$ is given by 
\[
\Big(~ h:C \longrightarrow T,~ p: T \longrightarrow X,~ q_1,\ldots,q_{l(\mu)}    ~\Big)
\]
where $(C, q_i)$ is a $l(\mu)$-marked prestable curve, $h$ is a genus $g$ stable map sending $q_i$ to $t$ and $p$ is a morphism sending $t$ to $x$ such that:
\begin{enumerate}
\item There is an equality of divisors on $C$ given by $h^{-1}(t) = \sum \mu_i q_i$.
\item We have $p|_{X}$ is an isomorphism and $p|_{T\setminus X}:T\setminus X\rightarrow \{x\}$ is constant. 
\item The pre-image of each node $n$ of $T$ is a union of nodes of $C$. At any such node $n'$ of $C$, the two branches of $n'$ map to the two branches of $n$, and their orders of branching are the same.
\item The data has finitely many automorphisms (Recall, an automorphism is a a pair of isomorphisms $a:C\rightarrow C$ and $b:T\rightarrow T$ taking $q_i$ to $q_i$ and $t$ to $t$ such that $h\circ a = b\circ h$ and $p = p\circ b$). 
\end{enumerate}

We denote by $\Mbar_g(X,\mu)$ the moduli stack of genus $g$ stable maps relative to $(\mu,x)$. This is the groupoid containing the objects: 

\[
\xi  = \Big(~
\begin{array}{c}\xymatrix@=1em{{C}\ar[d]_{\pi} \\ {S} \ar@/_/[u]_{q_i}}\end{array}
,~~ \begin{array}{c}\xymatrix@=1em{{T}\ar[d]_{\pi'} \\ {S} \ar@/_/[u]_{t}}\end{array}
,~~h: C \rightarrow T
,~~ p :T \rightarrow X
~\Big)
\]

where $\pi$ and $\pi'$ are flat proper morphisms, $h$ is a morphism over $S$ and for each geometric point $z\in S$ we have $\xi_z$ is a genus $g$ stable map relative to $(\mu,x)$. Furthermore, we require that in a neighbourhood of a node of $C_z$ mapping to a singularity of $T_z$ we can choose \'{e}tale-local coordinates on $S$, $C$ and $T$ with charts of the form $\Spec\,R$, $\Spec\,R[u,v]/(uv-a)$ and $\Spec\,R[x,y]/(xy-b)$ respectively such that the map is of the form $x \mapsto \alpha u^k$ and $y \mapsto \alpha v^k$ with $\alpha$ and $\beta$ units. A morphism $\xi_1 \rightarrow \xi_2$ in $\Mbar_g(X,\mu)$ between two appropriately label objects is a pair of cartesian diagrams
\[
\begin{array}{c}
\xymatrix@=1.5em{
 C_1\ar[d]_{\pi_1}\ar[r]^{a'} &  C_2\ar[d]^{\pi_2}\\
S_1 \ar[r]^{a} & S_2
}\end{array}
\hspace{1.5cm}
\begin{array}{c}
\xymatrix@=1.5em{
T_1\ar[d]_{\pi'_1}\ar[r]^{b'} &  T_2\ar[d]^{\pi'_2}\\
S_1 \ar[r]^{b} & S_2
}\end{array}
\]
that are compatible with the other data (i.e. we have $a'\circ q_{1,i} = q_{2,i} \circ a$, $b'\circ t_{1} = t_{2} \circ b$, $b'\circ h_1 = h_2\circ a'$ and $p_1 = p_2\circ b'$).

\subsection{The Canonical Ramification Section}\label{canonical_ram_subsection}

As we saw in the introduction, for a moduli point $[f]\in \Mbar_{g}(X,d)$ we have two natural morphisms
\begin{align}
f^*\omega_{X} \longrightarrow \Omega_{C} 
\hspace{1cm}\mbox{and}\hspace{1cm}
\Omega_{C} \longrightarrow \omega{c} \label{ramification_section_original_morphisms}
\end{align}
which we can combine into a single morphism $\delta: \O_C \rightarrow \omega_C\otimes f^* \omega_X^\vee$. This morphism reflects the ramification properties of $f$ which we will see in section \ref{br_section}. Hence, we will call the bundle $\omega_C\otimes f^* \omega_X^\vee$ the ramification bundle of $f$. \\

Considering the universal curve $\bm{\pi}: \overline{\calC}_g(X,d) \rightarrow \Mbar_g(X,d)$. The above construction still hold for the universal stable map $\bm{f}:\overline{\calC}_g(X,d) \rightarrow X$. We then have a universal section
\begin{align}\label{canonical_ram_section_def}
\bm{\delta} : \O_{\overline{\calC}_g(X,d)} \longrightarrow \overline{\calR} 
\end{align}
where we have denoted the universal ramification bundle $\overline{\calR} := \omega_{\bm{\pi}} \otimes \bm{f}^* \omega_X^\vee$.\\

For above case of stable maps we are interested in a subspace where a generic point $[f]$ corresponds to a map $f$ with a ramification divisor of the form  $R_f = r\cdot z_1 + \cdots + r \cdot z_m$ for disjoint points $z_1, \ldots ,z_m$. However, the key concept of relative stable maps is that the ramification above a fixed point is determined by a given divisor. The ramification is allowed to be free elsewhere.\\

Hence for the relative case we will be interested in a subspace of $\Mbar_g(X,\mu)$ where a generic $[f]$ corresponds to a map $f$ with a ramification divisor of the form:
 \[
 R_f = D_\mu  +  r\cdot z_1 + \cdots + r \cdot z_m
 \]
where $D_\mu = \sum(\mu_i-1)q_i$ is the ramification divisor supported at the points $q_i$ mapping to $x\in X$. So, we are interested in taking $r$th roots of a section of the bundle $\omega_C\otimes f^* \omega_X^\vee \otimes\O_C (-D_\mu) \cong \omega_C^{\log}\otimes f^* (\omega^{\log}_X)^\vee $. The situation is slightly more complicated because of the possibility of a degenerated target. So we consider a general genus $g$ stable map to $X$ relative to $(\mu, x)$ over $S$:
\[
\xi  = \Big(~
\begin{array}{c}\xymatrix@=1em{{C}\ar[d]_{\pi} \\ {S} \ar@/_/[u]_{q_i}}\end{array}
,~~ \begin{array}{c}\xymatrix@=1em{{T}\ar[d]_{\pi'} \\ {S} \ar@/_/[u]_{t}}\end{array}
,~~h: C \rightarrow T
,~~ p :T \rightarrow X
~\Big).
\]
Now, letting $q=q_1+\cdots+q_{l(\mu)} $, we have three line bundles which we are interested in:
\[
\omega_{C/S}^{\log} = \omega_{C/S}(q ),
\hspace{0.75cm}
\omega_{T/S}^{\log} = \omega_{T/S}( t)
\hspace{0.5cm}\mbox{and}\hspace{0.5cm}
\omega_{X}^{\log} = \omega_{X}(x)
\]
and we make choices of morphisms defining the divisors $q$, $t$ and $x$ respectively:
\begin{align}
D_q:\O_C(-q )\rightarrow \O_C,
\hspace{0.75cm}
D_t:\O_T(-t )\rightarrow \O_T
\hspace{0.5cm}\mbox{and}\hspace{0.5cm}
D_x:\O_X(-x )\rightarrow \O_X. \label{divisor_choices}
\end{align}
Now there is a unique choice of isomorphism $p^*\omega_{X}^{\log} \overset{\sim}{\rightarrow} \omega_{T/S}^{\log}$ such that the following diagram commutes
\[
\xymatrix{ 
 p^*\omega_{X} \ar[r]^{p^*D_x}\ar[d] & p^*\omega_{X}^{\log}  \ar[d]^\cong\\
 \omega_{T/S}  \ar[r]^{D_t} & \omega_{T/S}^{\log}
}
\]
where the left vertical morphism is the natural morphism coming from (\ref{ramification_section_original_morphisms}) applied to $p: T\rightarrow X$.\\

After using the isomorphism $p^*\omega_{X}^{\log} \overset{\sim}{\rightarrow} \omega_{T/S}^{\log}$ we are interested in a canonical morphism $h^* \omega_{T/S}^{\log} \longrightarrow \omega_{C/S}^{\log}$. The construction used in (\ref{canonical_ram_section_def}) breaks down here because there we used the fact that $\Omega_X \cong \omega_X$ is a locally free. In general, $\Omega_{T/S} \ncong \omega_{T/S}$ and $\Omega_{T/S}$ is not locally free, because of the nodes on the degenerated target. However, the admissibility condition allows us to  define a morphism $h^* \omega_{T/S} \longrightarrow \omega_{C/S}$ directly.\\

Away from the nodes of $T$ we can simply define the morphism in the usual way. Locally at the nodes we have that $S = \Spec\,R$, $T=\Spec\,R[x,y]/(xy-\xi)$ and $C=\Spec\,R[u,v]/(uv-\zeta)$ with the map $h$ defined by
\[\begin{array}{cccc}
H:&R[x,y]/(xy-\xi)& \longrightarrow & R[u,v]/(uv-\zeta) \\
&x & \longmapsto & \alpha u^a \\
&y & \longmapsto & \beta v^a 
\end{array}\]
for $\alpha$ and $\beta$ units and with $H(\xi)=\alpha\beta\zeta^a$. Also, locally we have that $\omega_{T/X}$  and $\omega_{C/X}$ are generated by
\[
\frac{dx\wedge dy}{(xy - \xi)} 
\hspace{1cm}\mbox{and}\hspace{1cm}
\frac{du\wedge dv}{(uv - \zeta)} 
\]
respectively. Hence, we have a natural isomorphism locally defined by:
\[
\frac{d(H(x))\wedge d(H(y))}{\big(H(x)H(y)  - \Phi(\xi)\big)} = \frac{d(\alpha u^a)\wedge d(\beta v^b)}{\big(\alpha\beta u^a v^a  - \alpha \beta \xi^a\big)} = \frac{du\wedge dv}{(uv - \zeta)}. 
\]
Hence we have the following lemma.

\begin{lemma}\label{lemma_ramification_morphism_degerated_target}
Let $T^{\mathrm{sm}}$ be the smooth locus of $T$ relative to $S$ and $B=h^{-1}(T^{\mathrm{sm}})$. There is a  canonical morphism $\widetilde{\delta}:h^*\omega_{T/S} \longrightarrow  \omega_{C/S}$ such that:
\begin{enumerate}
\item The restriction $\widetilde{\delta}|_{B}$ to is the usual morphism $(h|_B)^*\omega_{T^{\mathrm{sm}}/S} \rightarrow \omega_{C/S}$,
\item $\widetilde{\delta}$ is locally an isomorphism at the nodes of $T$. 
\end{enumerate}
\end{lemma}

Now, we restrict the morphism $h$ to the smooth locus of $C$ over $S$ and denote these by $h^{\mathrm{sm}}$
and $C^{\mathrm{sm}}$ respectively. The morphism from lemma \ref{lemma_ramification_morphism_degerated_target} restricted $C^{\mathrm{sm}}$ is injective and is the divisor sequence for the ramification divisor. Both the divisors $q_1+\cdots +q_{l(\mu)}$ and $h^{-1}(t)$ are in $C^{\mathrm{sm}}$. Now using the choices from (\ref{divisor_choices}) it is straightforward to show that there is now a unique map $\widetilde{\delta}^{\,\log}$ making the following diagram commute
\[
\xymatrix{ 
 h^*\omega_{T/S} \ar[r]^{h^*D_t}\ar[d]^{\widetilde{\delta}} & h^*\omega_{T/S}^{\log}  \ar[d]^{\widetilde{\delta}^{\,\log}}\\
 \omega_{C/S}  \ar[r]^{D_q} & \omega_{C/S}^{\log} 
}
\]
Now, using the isomorphism $p^*\omega_{X}^{\log} \overset{\sim}{\rightarrow} \omega_{T/S}^{\log}$ we have the canonical morphism which we desire:
\begin{align}\label{canonical_ram_section_def_log}
\delta^{\log} : \O_{C/S} \longrightarrow \omega_{C/S}^{\log}  \otimes f^*(\omega_{T/S}^{\log})^\vee. 
\end{align}

The above construction immediately leads itself to a universal construction. Considering the universal curve $\bm{\pi}: \overline{\calC}_g(X,\mu) \rightarrow \Mbar_g(X,\mu)$, the universal degenerated target $\bm{\pi'}:\calT \rightarrow \Mbar_g(X,\mu)$ with universal maps $\bm{h}:\overline{\calC}_g(X,\mu) \rightarrow \calT$ and $\bm{p}:\calT \rightarrow X$, and universal sections $\bm{q_i}:\Mbar_g(X,\mu)\rightarrow \overline{\calC}_g(X,\mu) $ and $\bm{t}:\Mbar_g(X,\mu)\rightarrow \calT$. We make, once and for all, choices 
\begin{align}
\bm{D_q}:\O_{\overline{\calC}_g(X,\mu)}(-\bm{q} )\rightarrow \O_{\overline{\calC}_g(X,\mu)},
\hspace{0.5cm}
\bm{D_t}:\O_{\calT}(-\bm{t} )\rightarrow \O_{\calT}
\hspace{0.4cm}\mbox{and}\hspace{0.4cm}
D_x:\O_X(-x )\rightarrow \O_X 
\end{align}
which allow us to define the universal section
\begin{align}\label{canonical_ram_section_def_relative}
\bm{\delta^{\log}} : \O_{\overline{\calC}_g(X,\mu)} \longrightarrow \overline{\calR}^{\log}
\end{align}
where we have denoted the universal ramification bundle $\overline{\calR}^{\log} := \omega_{\bm{\pi}}^{\log}  \otimes \bm{f}^*(\omega_{\bm{\pi'}}^{\log})^\vee$.\\

\subsection{\texorpdfstring{$r$}{r}-Prestable curves}\label{r-prestable_curves_section}

Our moduli problem requires the use of nodal curves where the nodes have a balanced $r$-orbifold structure. These curves are also called twisted curves and were introduced in \cite{AbraVist} to study stable maps where the target is a DM stack. They have since been extensively studied in \cite{Twisted_and_Admissible,Olsson,GW_of_DM,FJR1,FJR2}. In this paper we are interested in using them in relation to taking $r$th roots of line bundles, which have been studied in \cite{AbraJarvis,Chiodo_StableTwisted,Chiodo_Towards}.

\begin{definition}
Let $S$ be a scheme. An \textit{$r$-prestable curve} over $S$ of genus $g$ with $n$ markings is:
\[
\Big(~
\smallfamily{C}{S}{\pi}
,~~\Big( \begin{array}{c}\xymatrix@=1em{C   \\ S \ar[u]^{x_{i}} }\end{array}  \Big)_{i\in\{1,\ldots,n\}}
~\Big)
\]
where
\begin{enumerate}
\item $\pi$ is a proper flat morphism from a tame stack to a scheme;
\item each $x_i$ is a section of $\pi$ that maps to the smooth locus of $C$,
\item the fibres of $\pi$ are purely one dimensional with at worst nodal singularities,
\item the smooth locus $C^{\mathrm{sm}}$ is an algebraic space,
\item the coarse space $\overline{\pi}: \overline{C} \rightarrow S$ with sections $\overline{x}_i$ is a genus $g$, $n$-pointed prestable curve
	\[
	\big(~\overline{C},~~ \overline{\pi}:\overline{C} \rightarrow S,~~ (\overline{x}_i: S\rightarrow \overline{C})_{i\in\{1,\ldots,n\}} ~\big)
	\]
\item the local picture at the nodes is given by $\left[ U/ \mu_r\right] \rightarrow T$, where 
	\begin{itemize}
	\item $T = \Spec A$, $U = \Spec A[z,w]/(zw-t)$ for some $t\in A$, and the action of $\mu_r$ is given by $(z,w)\mapsto (\xi_r z, \xi_r^{-1} w)$.
	\end{itemize}
\end{enumerate}
\end{definition}

We denote the space parameterising $r$-prestable curves by $\frakM^r_{g,n}$. This space is shown to be a smooth proper stack in \cite{Chiodo_StableTwisted}. There is a natural forgetful map 
\[
\frakM^r_{g,n} \longrightarrow \frakM_{g,n}
\]
which takes an $r$-prestable curve to its coarse space. This map is flat and surjective of degree 1, but it is not an isomorphism. In fact its restriction to the boundary is degree $\frac{1}{r}$. \\

One can also consider $r$-orbifold structure at smooth marked points as well. Specifically we can include in the definition \'{e}tale gerbes $\calX_i \rightarrow X$ which are closed sub-stacks of the smooth locus of the curve $\calX_i \hookrightarrow C^{\mathrm{sm}}$. The local picture at an $r$-orbifold marked point is given by $\left[ V/ \mu_r\right] \rightarrow T$, where 
	\begin{itemize}
	\item[] $T = \Spec A$, $V = \Spec A[z]$, and the action of $\mu_r$ is given by $z\mapsto \xi_r z$.\\
	\end{itemize}

\subsection{Line Bundles and their \texorpdfstring{$r$}{r}th Roots on \texorpdfstring{$r$}{r}-Twisted Curves}

The theory of $r$th roots on line bundles on prestable curves has origins related to theta characteristics \cite{Corn_theta}. This lead naturally to the study of $r$-spin structures studied in \cite{JarvisTorsionFree, JarvisGeomHigh} using torsion free sheaves for $r$th roots. $r$-spin structures were also studied using twisted curves in \cite{AbraJarvis,Chiodo_StableTwisted, Chiodo_Towards} and other methods in \cite{CCC}. The results of \cite{Chiodo_StableTwisted} will be of particular interest to us.\\

Consider a family of $r$-prestable curves $C\overset{\gamma}{\rightarrow} \overline{C} \overset{\overline{\pi}}{\rightarrow} S$, and let $E$ be a line bundle on $C$ pulled back from $\overline{C}$ with relative degree divisible by $r$. Define the following groupoid $\mathrm{Root}_C^r ( E)$ containing as objects:
\[
					\Big(~
					h: Z\rightarrow S
					,~~ L,~~ e: L^r \overset{\sim}{\rightarrow}E_Z
					~\Big)
\]
where $h$ is a morphism of schemes, $L$ is a line bundle on $C_Z$ and $e$ is an isomorphism.

\begin{theorem}\label{chiodo_torsor_thm}\, \cite[Prop 3.7, Thm 3.9]{Chiodo_StableTwisted} In the situation above we have:
\begin{enumerate}
\item For each geometric point $p\in S$, we have $r^{2g}$ roots of $E_p$. 
\item $\mathrm{Root}_C^r ( \O_C)$ is a finite group stack. 
\item $\mathrm{Root}_C^r ( E)$ is a finite torsor under $\mathrm{Root}_C^r ( \O_C)$.
\item $\mathrm{Root}_C^r ( E) \rightarrow S$ is \'{e}tale of degree $r^{2g-1}$. \\
\end{enumerate}
\end{theorem}

\begin{re}
Consider an $r$-prestable curve over $\C$ with an $r$-orbifold marked point $p$. The local picture at the marking $p$ is given by $\left[ (\Spec\,\C[z])/ \mu_r\right] $ where the action of $\mu_r$ is given by $z\mapsto \xi_r z$. Consider a line bundle $L$ supported at $p$. Then $L$ is locally defined at $p$ by $\phi=z^n$ for $n\in \mathbb{Z}$ and we have $\phi(\xi_r z) = \xi_r^k \phi(z)$ for some $k \in \mathbb{Z}/r$. We call $k$ the \textit{multiplicity} of $L$ at $p$.\\

Similarly, at a node $q$ the the local picture is given by $\left[ (\Spec\,\C[u,v]/uv)/ \mu_r\right]$, where the action of $\mu_r$ is given by $(u,v)\mapsto (\xi_r u, \xi_r^{-1} v)$. So a line bundle $L$ on $C$ supported at $q$ is locally defined by $\psi = u^{n_1}-v^{n_2}$ for $n_1,n_2\in \mathbb{Z}\setminus \{0\}$ such that $\psi(\xi_r u, \xi_r v) = \xi_r^a \psi(u, v)$ for some $a \in \mathbb{Z}/r$. In fact $a$ is determined only up to a choice of branch. Hence we obtain a pair numbers $a,b\in \mathbb{Z}/r$ with either $a=b=0$ or $a+b=r$.  We call this pair the \textit{multiplicity} of $L$ at the node $q$.\\
\end{re}

We can also consider an associated sheaf on the coarse space $\overline{C}$. Locally the coarse space is given by the invariant sections of the structure sheaf. Then the sheaf $\overline{L}:=\gamma_* L$ is similarly given by the locally invariant sections of $L$. When $C$ is a smooth curve $\overline{L}$ is a line bundle. However when $C$ is singular, the $\overline{L}$ is only torsion free in general. Using these ideas one can easily show the following lemma is true.\\

\begin{lemma}\label{rth_root_pushforward_divisor_lemma}
Let $C$ be as $r$-prestable curve over $\C$, with $n$ smooth orbifold points $x_1,\ldots, x_n$ and let $\beta:C\rightarrow \mathsf{C}$ be the map forgetting the orbifold structure at the smooth points. Also let $L$ be a line bundle on $C$ with multiplicities $a_1, \ldots, a_n$ at the the orbifold points and let $D$ be the divisor $\sum a_i x_i$. Then for a section $\sigma: \O_C \rightarrow L$ there is a commuting diagram where the bottom row is the divisor sequence:
\[
\xymatrix{
 & \O_{\mathsf{C}} \ar@{=}[r] \ar[d]^{(\beta_* \sigma)^r} & \O_{\mathsf{C}}   \ar[d]^{\beta_* (\sigma^r)}&  & \\
0 \ar[r] & (\beta_* L)^r \ar[r] &\beta_* (L^r) \ar[r] &\O_D \ar[r] &0 
}
\]
\end{lemma}

\section{Stable Maps with Roots of Ramification}\label{construction_section}

\textbf{Section \ref*{construction_section} Notation:} Recall the \hyperlink{notation_main}{notation convention}:
\begin{itemize}
\item  $\M$ is either $\Mbar_{g}(X,d)$ or $\Mbar_{g}(X,\mu)$ for $g \geq 0$, $d>0$ and $\mu$ a partition of $d$.
\item $\calC\rightarrow \M$ be the associated universal curve. 
\item  If $\M$ is $\Mbar_{g}(X,d)$ (resp. $\Mbar_{g}(X,\mu)$) then $\frakM$ is $\frakM_g$ (resp. $\frakM_{g,l(\mu)}$).
\item The expected number of order $r$ special loci in the generic case is denoted by $m$. When $\M=\Mbar_g(X,d)$  we have $m = \frac{1}{r} (2g-2-d(2g_X-2))$ and when $\M=\Mbar_g(X,\mu)$ we have $m = \frac{1}{r}  (2g-2+l(\mu)+|\mu|(1-2g_X))$.
\item Whenever we introduce a new related space, the notation will carry through to the original spaces. For $\Mbar_{g}(X,\mu)$ the two key spaces are:
\begin{itemize}
\item $\M^{[r]} = \Mbar^{\,[r]}_{g}(X,\mu)$ and $\calC^{[r]} = \overline{\calC}^{\,[r]}_{g}(X,\mu)$ defined in \ref{def_stable_maps_with_roots_of_the_ramification_bundle}.\\
\item $\M^{[\frac{1}{r}]} = \Mbar^{_{[1/r]}}_{g}(X,\mu)$ and $\calC^{[\frac{1}{r}]} = \overline{\calC}^{_{[1/r]}}_{g}(X,\mu)$ are defined in \ref{main_moduli_spaces_definitions}.\ref{full_space_part}.\\
\end{itemize}
And the associated spaces are:
\begin{itemize}
\item $\M^{r} = \Mbar^{\,r}_{g}(X,\mu)$ and $\calC^{r} = \overline{\calC}^{\,r}_{g}(X,\mu)$ defined in \ref{moduli_with_roots_of_ramification_section}.\\
\item $\M^{{\frac{1}{r}, \calE}} = \Mbar^{_{\frac{1}{r}, \calE}}_{g}(X,\mu)$ and $\calC^{{\frac{1}{r}, \calE}} = \overline{\calC}^{_{\frac{1}{r}, \calE}}_{g}(X,\mu)$ are defined in \ref{roots_of bundle_moduli_definition}.\\
\item $\M^{\frac{1}{r}} = \Mbar^{_{1/r}}_{g}(X,\mu)$ and $\calC^{1/r} = \overline{\calC}^{_{1/r}}_{g}(X,\mu)$ are defined in \ref{main_moduli_spaces_definitions}.\ref{reduced_space_part}.\\
\end{itemize}
\end{itemize}

\subsection{\texorpdfstring{$r$}{r}-Stable Maps with Roots of the Ramification Bundle}\label{moduli_with_roots_of_ramification_section}

In this subsection we will consider the results of \cite{Chiodo_StableTwisted} in the context of stable maps.  We will begin by considering stable maps where the domain curve is $r$-prestable. We call these $r$-stable maps. The moduli stack of these and its universal curve fit into the two cartesian squares:
\[\xymatrix{
\calC^r \ar[r]\ar[d]_{\bm{\gamma}} &\M^r\ar[r]\ar[d]^{} &  \ar[d]^{} \mathfrak{M}^{r}\\
\calC \ar[r] &\M \ar[r] & \mathfrak{M} 
}\]

Now we will considering stable maps with an $r$th root of a line bundle. Let $\overline{\calE}$ be a line bundle on $\calC$  of degree divisible by $r$ and define the line bundle $\calE$ on   $\calC^r $ by $\calE := \bm{\gamma}^* \overline{\calE}$.

\begin{definition} \label{roots_of bundle_moduli_definition}
Denote by $\M^{\frac{1}{r}, \calE}$ the moduli stack of \textit{$r$-stable maps with roots of $\calE$} which contains families:
				\[
					\Big(~
					\xi
					,~~ \L,~~ e: \L^r \overset{\sim}{\longrightarrow} \calE_\xi ~\Big)
				\]
				where
				\begin{enumerate}
					\item $\xi$ is a family of $r$-stable maps in $\M$;
					\item $\L$ is a line bundle on $\calC_{\xi}$; 
					\item $e$ is an isomorphism of line bundles on $\calC_{\xi}$. 
				\end{enumerate}
				
\end{definition}

\begin{lemma}\label{mod_div_ram_construction} The $\M^{\frac{1}{r}, \calE}$ has the following properties:
\begin{enumerate}
\item
$\M^{\frac{1}{r}, \calE}$ is a proper DM stack.
\item
When $\calE$ is the trivial bundle  $\M^{\frac{1}{r}, \O_{\calC}}\rightarrow \M^r $ is a finite group stack.
\item
The forgetful map  $\M^{\frac{1}{r}, \calE}  \rightarrow\M^r$ is a finite torsor under $\M^{\frac{1}{r}, \O_{\calC}}$ and is \'{e}tale of degree $r^{2g-1}$. 
\end{enumerate}
\end{lemma}

\begin{proof}
Let $a: S \rightarrow  \M^r$ be the morphism of stacks defined by the family $\xi \in \M^r$. Also let $C = (\calC^r)_{\xi}$ and $E= (\calE)_{\xi}$. Then we have the following cartesian diagrams:
\[
\xymatrix{
\mathrm{Root}_C^r(\O_C) \ar[r]\ar[d] & \M^{\frac{1}{r},\O_{\calC}} \ar[d]\\
S \ar[r] & \M^r
}
\hspace{1cm}
\xymatrix{
\mathrm{Root}_C^r(E) \ar[r]\ar[d] & \M^{\frac{1}{r},\calE} \ar[d]\\
S \ar[r] & \M^r
}\]
The lemma now follows from theorem \ref{chiodo_torsor_thm}. 
\end{proof}

\begin{definition}\label{def_stable_maps_with_roots_of_the_ramification_bundle}
In the special case where $\calE = \calR$ we call $\M^{\frac{1}{r},\calR}$ the  moduli space of \textit{$r$-stable maps with roots of the ramification bundle} and denote it with the simplifying notation:
\[
\M^{[r]}:= \M^{\frac{1}{r},\calR}. 
\]
\end{definition}

\subsection{Power Map of Abelian Cone Stacks} 

Let $\overline{\calE}$ be a line bundle on $\calC$  of degree divisible by $r$ and define the line bundle $\calE$ on   $\calC^{[r]}$ by $\calE := \bm{\gamma}^* \overline{\calE}$ where $\bm{\gamma}: \calC^r \rightarrow \calC$ is the map forgetting the $r$-orbifold structure of the curves. Consider $\M^{\frac{1}{r}, \calE}$, the space of $r$-stable maps with roots of $\calE$ defined in \ref{roots_of bundle_moduli_definition} with universal curve $\bm{\pi}: \calC^{\frac{1}{r}, \calE} \rightarrow \M^{\frac{1}{r}, \calE}$, universal section $\bm{s}: \M^{\frac{1}{r}, \calE} \rightarrow \calC^{\frac{1}{r}, \calE}$ and universal $r$th root bundle $\L$ and isomorphism $\bm{e}:\L^r\overset{\sim}{\rightarrow} \calE$.\\

\begin{definition}\label{totalspace_definition}
For a line bundle $\calF$ on $\calC^{\frac{1}{r}, \calE}$, we define the the following notation:
\begin{enumerate}
\item $\mathsf{Tot}\,\bm{\pi}_*\calF := \Spec_{\M^{\frac{1}{r}, \calE}}\, \Big(\Sym^\bullet \,R^1\bm{\pi}_*(\calF^\vee \otimes \omega_{\bm{\pi}} ) \Big)$ which contains objects:
\[
					\Big(~
					\xi
					,~~
					\sigma: \O_{C} \longrightarrow \calF_{\xi} ~\Big)
\]
where $\xi$ is an object of $\M^{\frac{1}{r}, \calE}$ and $C:= (\calC^{\frac{1}{r}, \calE})_{\xi}$ (discussed in \cite[Prop 2.2]{ChangLi} and \cite[Thm 2.11]{ChangLiLi}). Also, let $\bm{\alpha}: \mathsf{Tot}\,\bm{\pi}_*\calF \rightarrow \M^{\frac{1}{r}, \calE}$ denote the natural forgetful map.

\item $\bm{\psi}: \calC_{\mathsf{Tot}\,\bm{\pi}_*\calF}\rightarrow \mathsf{Tot}\,\bm{\pi}_*{\calF}$ is the universal curve and $\bm{\widehat{\alpha}}: \calC_{\mathsf{Tot}\,\bm{\pi}_*\calF} \rightarrow \calC^{\frac{1}{r}, \calE}$ is the natural forgetful map.

\item $\mathsf{Tot}\,\calF := \Spec_{ \calC^{\frac{1}{r}, \calE}}\, \Big(\Sym^\bullet \, \calF^\vee \Big)$ which contains objects:
\[
					\Big(~
					\zeta
					,~~
					\lambda: \O_{S} \longrightarrow s^*\calF_{\zeta} ~\Big)
\]
where $\zeta$ is an object of $ \calC^{\frac{1}{r}, \calE}$ over $S$ and $s:=\bm{s}_\zeta$. Also, let $\bm{\check{\alpha}}: \mathsf{Tot}\,\calF \rightarrow \calC^{\frac{1}{r}, \calE}$ denote the natural forgetful map.
\end{enumerate}
\end{definition}

\begin{remark}
Note that while we use the notation $\mathsf{Tot}\,\bm{\pi}_*\calF$, it is often the case that $\bm{\pi}_*\calF$ is not locally free. This space is called an \textit{abelian cone} in \cite{Intrinsic}. \\
\end{remark}

 Let $\zeta$ be a family in $ \calC^{\frac{1}{r}, \calE}$ over $S$ with $\bm{\pi}_{\zeta} = \pi: C\rightarrow S$ and $s:=\bm{s}_\zeta$. There is a natural evaluation morphism $\mathfrak{e} : \calC_{\mathsf{Tot}\,\bm{\pi}_*\calF} \rightarrow \mathsf{Tot}\,\calF$ defined by
 \begin{align}
					\mathfrak{e} ~:~ \Big(~
					\zeta
					,~~
					\sigma: \O_{C} \longrightarrow \calF_{\xi} ~\Big)
					\longmapsto
					\Big(~
					\zeta
					,~~
					s^*\sigma: \O_{S} \longrightarrow s^*\calF_{\zeta} ~\Big). \label{evaluation_map_definition}
\end{align}

This gives the following commutative diagram where the left-most square is Cartesian:
\begin{align}\label{main_cone_squares}\begin{array}{c}
\xymatrix{
\mathsf{Tot}\,\bm{\pi}_*\calF \ar[d]^{\bm{\alpha}} & \calC_{\mathsf{Tot}\,\bm{\pi}_*\calF} \ar[l]_{\bm{\psi}} \ar[r]^{\mathfrak{e}}\ar[d]^{\bm{\widehat{\alpha}}}& \mathsf{Tot}\,\calF \ar[d]^{\bm{\check{\alpha}}} \\
\M^{\frac{1}{r}, \calE} & \calC^{\frac{1}{r}, \calE} \ar[l]^{\bm{\pi}} \ar@{=}[r] &\calC^{\frac{1}{r}, \calE}
}\end{array}
\end{align}\\

There are special cases when $\calF = \L$ and $\calF = \L^r$ and we have we have canonical maps:
\begin{definition}\label{power_map_definition}
The \textit{$r$th power map over $\M^{\frac{1}{r}, \calE} $} is the map $\tau: \mathsf{Tot}\,\bm{\pi}_*\L \rightarrow \mathsf{Tot}\,\bm{\pi}_*\L^r$ defined by:
\[
					\Big(~
					\xi
					,~~
					\sigma	 ~\Big)
					\longmapsto
					\Big(~
					\xi
					,~~
					\sigma^r ~\Big)
\]
and the \textit{$r$th power map over $\calC^{\frac{1}{r}, \calE} $} is the similarly defined map $\bm{\check{\tau}}: \mathsf{Tot}\,\L \rightarrow \mathsf{Tot}\,\L^r$. 
\end{definition}

Out of the two maps defined here, $\bm{\check{\tau}}$ is nicer. It is a fibre-wise $r$-fold cover of the total space of $\mathsf{Tot}\,\L^r$ ramified at the zero section. However, $\tau$ is the map more directly related to out moduli problem.

\begin{lemma}\label{rth_powermap_factor_lemma}
The $r$th power map over $\M^{\frac{1}{r}, \calE} $ is factors via
\[
\xymatrix@R=1.75em{
\mathsf{Tot}\,\bm{\pi}_*\L \ar[rr]^{\tau} \ar[rd]_{\varphi} &&\mathsf{Tot}\,\bm{\pi}_*\L^r\\
& \calX \ar[ru]_{j} &
}
\]
where $j$ is a closed immersion and $\varphi$ is the quotient by the following action of $\mathbb{Z}_r$ on $\mathsf{Tot}\,\bm{\pi}_*\L$: 
\[
					\zeta_r \cdot \Big(~
					\xi
					,~~
					\sigma	 ~\Big)
					=
					\Big(~
					\xi
					,~~
					\zeta_r \cdot \sigma	 ~\Big)
\]
\end{lemma}

\begin{proof}
We first show that the image of $\tau$ is a closed substack of $\mathsf{Tot}\,\bm{\pi}_*\L^r$. Denote the closed immersion defined by taking the graph of $\tau$ by $i: \mathsf{Tot}\,\bm{\pi}_*\L \rightarrow \mathsf{Tot}\,\bm{\pi}_*\L^r \times_{\M^{\frac{1}{r}, \calE} } \mathsf{Tot}\,\bm{\pi}_*\L $. Then $\tau$ factors via:
\[
\xymatrix@R=1.75em{
\mathsf{Tot}\,\bm{\pi}_*\L \ar@{^{(}->}[r]^{i\hspace{0.75cm}} \ar[rd]_{\tau} &\mathsf{Tot}\,\bm{\pi}_*\L^r \times_{\M^{\frac{1}{r}, \calE} } \mathsf{Tot}\,\bm{\pi}_*\L \ar[d]^{\mathrm{pr}_1}  \\
&\mathsf{Tot}\,\bm{\pi}_*\L^r 
}
\]

We claim that $\mathrm{pr}_1$ is a closed map. To see this we let $\bm{\psi}: \calC_{\mathsf{Tot}\,\bm{\pi}_*\L^r} \rightarrow \mathsf{Tot}\,\bm{\pi}_*\L^r$ be the universal family. Then we have the following abelian cone stack over $\mathsf{Tot}\,\bm{\pi}_*\L^r$ 
\[
p:\Spec_{\mathsf{Tot}\,\bm{\pi}_*\L^r}\Big(\Sym^\bullet R^1{\psi}_*(\psi^*\L^\vee \otimes \omega_{\psi} ) \Big){\longrightarrow} \mathsf{Tot}\,\bm{\pi}_*\L^r
\]
which is isomorphic over $\mathsf{Tot}\,\bm{\pi}_*\L^r$ to the pullback:
\[
\xymatrix@C=0.5em{
\mathsf{Tot}\,\bm{\pi}_*\L^r \times_{\M^{\frac{1}{r}, \calE}} \mathsf{Tot}\,\bm{\pi}_*\L \ar[rr]^{\sim\hspace{0.75cm}} \ar[rd]_{\mathrm{pr}_1} &&  \Spec_{\mathsf{Tot}\,\bm{\pi}_*\L^r} \Big(\Sym^\bullet R^1{\psi}_*(\psi^*\L^\vee \otimes \omega_{\psi} )\Big) \ar[dl]^{p}\\
& \mathsf{Tot}\,\bm{\pi}_*\L^r &
}
\]
$p$ is a closed map, so $\mathrm{pr}_1$ is also closed and $\mathrm{im}(\tau)$ is a well defined closed substack. It is clear that $\mathrm{im}(\tau)$ is isomorphic to the quotient of $\mathsf{Tot}\,\bm{\pi}_*\L$ by the action of $\mathbb{Z}_r$. 
\end{proof}

\subsection{Proof of Theorem \ref{construction_and_proper_theorem}}

In this section we will prove theorem \ref{construction_and_proper_theorem} about the properties of $\M^{{1/r}}$. We will also consider related spaces that contain extra information which we will denote by $\M^{[1/r]}$ and $\M^{(1/r)}$. In particular $\M^{[1/r]}$ will be the key space for study in the later sections of this article. \\

Let $\bm{\pi}: \calC^{[r]} \rightarrow \M^{[r]}$ be the universal curve of $\M^{[r]}$ and $\bm{f}: \calC^{[r]} \rightarrow X$ be the universal $r$-stable map and $\bm{\delta}: \O_{\calC^r} \rightarrow \calR$ be the pullback by $\bm{\gamma}: \calC^r \rightarrow \calC$ of canonical ramification section defined in (\ref{canonical_ram_section_def}) and (\ref{canonical_ram_section_def_relative}). Where $\bm{\gamma}$ is the map which forgets the $r$-orbifold structure.  Also let $\L$ be the universal $r$th root on $\M^{[r]}$.  \\

\begin{definition} \label{main_moduli_spaces_definitions}
The moduli spaces of stable maps with divisible ramification are:
\begin{enumerate}
\item \label{reduced_space_part}
$\M^{{1/r}}$ is the substack of $\M^{{r}}$ containing families $\xi$ where there exists:
\begin{enumerate}
\item a line bundle $L$ on $C:= (\calC^r)_\xi$;
\item an isomorphism $e: L^r \overset{\sim}{\rightarrow} \calR_\xi$;
\item a morphism $\sigma: \O_C \rightarrow L$;
\end{enumerate}
such that $e(\sigma^r) = \bm{\delta}_\xi$. 

\item $\M^{(1/r)}$ is the substack of $\M^{[r]}$ containing families $\zeta = (\xi, L, e)$ where $\xi$, $L$ and $e$ are as above and there exists a morphism $\sigma: \O_C \rightarrow L$ as above. 

\item\label{full_space_part} $\M^{[1/r]}$ is the substack of $\mathsf{Tot}\,\bm{\pi}_*\L$ containing families $\chi = (\xi, L, e, \sigma)$ where $\xi$, $L$, $e$ and $\sigma$ are as above.

\end{enumerate}
\end{definition}

 These three stacks are related by the following diagram where the horizontal arrows are forgetful maps and the vertical arrows are inclusions. 
 \[
\xymatrix{
\M^{[\frac{1}{r}]} \ar[d] \ar[r] & \M^{(\frac{1}{r})} \ar[d] \ar[r] &\M^{\frac{1}{r}} \ar[d] \\
\mathsf{Tot}\,\bm{\pi}_*\L \ar[r] &\M^{[r]} \ar[r] & \M^r
}
\]\\

After pulling back to $\M^{[r]}$, the canonical ramification section $\bm{\delta} :\O_{\calC^r} \rightarrow \calR$  and the universal $r$th root $\bm{e}:\L^r\overset{\sim}{\rightarrow} \calR$ define a natural inclusion:
\begin{align}
\begin{array}{cccc}
\bm{i'} : & \M^{[r]} & \longrightarrow & \mathsf{Tot}\,\bm{\pi}_*\L^r \\
& \xi &\longmapsto & \big(~\xi, \,\bm{e}_\xi^{-1}(\bm{\delta}_\xi)\,\big).
\end{array}\label{inclusion_definition_delta}
\end{align}

$\M^{[{1/r}]}$ now fits into the following cartesian diagram defining $\bm{\nu}$:
\[\xymatrix{
		\M^{[{\frac{1}{r}}]} \ar[r]^{\bm{i}\hspace{0.5em}}\ar[d]_{\bm{\nu}} &\mathsf{Tot}\,\bm{\pi}_*\L \ar[d]^{\bm{\tau}}\\
		\M^{[r]}\ar[r]_{\bm{i'}\hspace{0.5em}} &\mathsf{Tot}\,\bm{\pi}_*\L^r
}\]
Lemma \ref{rth_powermap_factor_lemma} shows that $\M^{[{1/r}]}$ is a proper $DM$ stack. We have that $\M^{(1/r)}$ is the quotient of $\M^{[{1/r}]}$ by the action of $\mathbb{Z}/r$, showing $\M^{(1/r)}$  is a closed substack of $ \M^{[r]}$. Also, since $\M^{[r]} \rightarrow \M^{r}$ is proper we can define $\M^{1/r}$ to be the closed substack of $\M^r$ coming from the image of $\M^{(1/r)}$. Hence we have proved theorem \ref{construction_and_proper_theorem} after composing with $\M^r \rightarrow \M$, which is flat and proper.\\

Note that the forgetful map $\M^{[1/r]} \rightarrow \M^{(1/r)}$ is \'{e}tale of degree $r$. However, the map $\M^{(1/r)} \rightarrow \M^{1/r}$ is more complicated and in general not \'{e}tale. There are cases where the map is \'{e}tale such as when the genus is zero, then the map is degree $1/r$. For another example, consider the space $\M^{(1/r)}_g(\P^1,1)$ where the map is \'{e}tale of degree $r^{2g-1}$.

\section{Branching and Ramification of \texorpdfstring{$r$}{r}-Stable Maps}\label{br_section}

\textbf{Section \ref*{br_section} Notation:} Consider the universal objects of $\M^{[1/r]}$:
\begin{enumerate}
\item The universal curve, $\bm{\rho}:\calC^{[1/r]}\rightarrow \M^{[1/r]}$
\item The universal stable map, $\bm{f}:\calC^{[1/r]}\rightarrow X$ and $\bm{F}:\calC^{[1/r]}\rightarrow X\times \M^{[1/r]}$
\item The universal canonical section, $\bm{\delta}: \O_{\calC^{[1/r]}} \rightarrow \calR$
\item The universal $r$th root of $\bm{\delta}$, $(\L, \bm{e}:\L^r\overset{\sim}{\rightarrow} \calR, \bm{\sigma}:  \O_{\calC^{[1/r]}} \rightarrow \L)$\\
\end{enumerate}

For a family $\xi$ over $S$ in $\M^{[1/r]}$ we will use the following notation
\begin{itemize}
\item[] $C:=\calC^{[1/r]}_\xi$, $ \rho:=\bm{\rho}_\xi$, $f:= \bm{f}_\xi $, $F := \bm{F}_\xi$, $\delta:=\bm{\delta}_\xi$, $L:=\L_\xi $, $ e:=\bm{e}_\xi$ and $ \sigma:=\bm{\sigma}_{\xi}$.\\
\end{itemize}

We denote the expected number of order $r$ ramification loci in the generic case by $m = \frac{1}{r} (2g-2-d(2g_X-2))$ in the case $\M=\Mbar_g(X,d)$ and  $m = \frac{1}{r}  (2g-2+l(\mu)+|\mu|(1-2g_X))$ in the case $\M=\Mbar_g(X,\mu)$. For a morphism of sheaves $a:A\rightarrow B$, we will denote the associated complex in degree $[-1,0]$ by $[a:A\rightarrow B]$.

\subsection{Divisor Construction and the Branch Morphism for Stable Maps}

As we saw in the introduction and discussed in \ref{canonical_ram_subsection} the ramification divisor is not well defined for stable maps. However it is possible to define a branch divisor using the canonical ramification section defined in   \ref{canonical_ram_subsection}. To do this, we must first review a construction of Mumford \cite[\S 5.3]{GIT} which allows us to assign a Cartier divisor to certain complexes of sheaves. \\

Let $Z$ be a scheme and recall that a complex of sheaves $E^\bullet$ is torsion if the support of each $\calH^i(E^\bullet)$ does not contain any of the associated points of $Z$. Let $E^\bullet$ be a finite torsion complex of free sheaves on $Z$ and let $U\subset Z$ be the complement of $\bigcup_i \Supp\,\calH^i(E^\bullet)$. Then $E^\bullet|_U$ is exact and $U$ contains all the associated points of $Z$. There are two ways to construct isomorphisms
\[
\mathrm{det}\, E^\bullet|_U \overset{\sim}{\longrightarrow} \O_U.
\]
\begin{enumerate}
\item $\kappa$: This is a canonical isomorphism which arises from the exactness of $E^\bullet|_U$.
\item $\Psi$: Which is from an explicit choice of isomorphism $E_i\overset{\sim}{\longrightarrow} \O_U$ for each $i$. \\
\end{enumerate}

So a choice of $\Psi$ defines a section $\Psi\circ \kappa^{-1}\in H^0(U, \O_U^*)$. Also, it is shown in \cite[Lemma 1]{FantechiPand} that if $U$ contains all the associated points of $Z$ then a section of $H^0(U, \O_U^*)$ defines a canonical section $\lambda$ of  $H^0(Z, \mathcal{K}^*)$. A different choice of $\Psi$ amounts to multiplication of $\lambda$ by an element of $H^0(Z, \O_Z^*)$. In this way $E^\bullet$ defines an element of $H^0(Z, \mathcal{K}^*/\O_Z^*)$. This construction also holds when $E^\bullet$ is a perfect complex (i.e. locally isomorphic to a finite complex of locally free sheaves). 

\begin{definition}
Let $E^\bullet$ be a perfect torsion complex. The divisor associated to $E^\bullet$ is the divisor constructed above and is denoted by $\div(E^\bullet)$. 
\end{definition}

The divisor construction has the following important properties. 
\begin{lemma} \cite[Prop 1]{FantechiPand} Let $E^\bullet$ be a perfect torsion complex.
\begin{enumerate}
\item $\div(E^\bullet)$ depends only on the isomorphism class of $E^\bullet$ in the derived category of $Z$.
\item If $F$ is a coherent sheaf admitting a finite free resolution $F^\bullet$, then we have $\div(F):= \div(F^\bullet)$ is an effective divisor. 
\item If $D$ is an effective divisor then $\div(\O_D) = D$. 
\item The divisor construction is additive on distinguished triangles. 
\item If $h:Z'\rightarrow Z$ is a base change such that $h^*E^\bullet$ is torsion, then we have $\div(h^*E^\bullet) = h^*\div(E^\bullet)$.
\item If $L$ is a line bundle then $\div(E^\bullet \otimes L) = \div(E^\bullet )$. 
\end{enumerate}
\end{lemma}

The divisor construction is use in \cite{FantechiPand} to construct a morphism 
\[
\begin{array}{cccc}
br : &\M &\longrightarrow &\Sym^{m'} X. 
\end{array}
\]
where $m'$ is the virtual dimension of $\M$. In particular, if $\zeta \in \M$ is a family of stable maps over $S$ and $\mathrm{pr}_2:X\times S\rightarrow S$ is the projection, then they show that the canonical ramification section (see \ref{canonical_ram_subsection}) defines a $\mathrm{pr}_2$-relative effective Cartier divisor of degree $m'$:
\[
B_\zeta:= \div\Big( R(\bm{F}_\zeta)_*[\O_{\calC_{\zeta}} \overset{\bm{\delta}_{\zeta}}{\longrightarrow} \overline{\calR}_{\zeta}]\Big).
\]
Hence the map $br$ is defined by $\zeta \mapsto B_\zeta$. 

\begin{remark}
In \cite{FantechiPand} the relative case $\M = \Mbar_g(X, \mu)$ is not considered. However, the results and proofs required to define $br$ work in this case when we use the section $\bm{\delta}$ constructed in section  \ref{canonical_ram_subsection} and  $m'=2g-2-d(2g_X-2)$ is replaced by $m'=2g-2+l(\mu)+|\mu|(1-2g_X)$. 
\end{remark}

\subsection{A Branch Morphism for Maps with Divisible Ramification}

We will show in this section that a branch morphism can be constructed for stable maps with divisible ramification. The role of the canonical ramification section will be replaced by its universal $r$th root.

\begin{lemma}
The direct image $RF_*[\O_C\overset{\sigma}{\longrightarrow} L]$ is a perfect torsion complex. 
\end{lemma}
\begin{proof}
Recall that $F$ factors via the forgetful map to the coarse space $F = \overline{F} \circ \gamma$ where $\overline{F}:= (\overline{f},\overline{\rho})$. Also recall that $\gamma_*$ is an exact functor, so we have the quasi-isomorphism $RF_*[\O_C\overset{\sigma}{\longrightarrow} L] \cong R\overline{F}_*[\O_{\overline{C}} \overset{\gamma_*\sigma}{\longrightarrow} \gamma_*L]$. $\overline{F}$ has finite tor-dimension, so $R\overline{F}_*[\gamma_*\sigma]$ is quasi-isomorphic to a finite complex of quasi-coherent sheaves on $X\times S$ flat over $S$. Denote this complex by $E^\bullet$.\\

Perfect is a local property so we can assume that $S=\Spec A$. Also, let $pr_1: X\times S\rightarrow X$ and $pr_2: X\times S\rightarrow S$ be the natural projections. Thus we have that $M:=pr_1^* \O_X(1)$ is an ample line bundle on the fibres of $pr_2$.  Then for sufficiently large $n$ we have for each $E_i$ the following properties:
\begin{enumerate}
\item $A_i :=\rho^* \rho_*( E_i\otimes M^{n})\otimes M^{\SmNeg n}$ is locally free. 
\item The natural map $a_i:A_i \longrightarrow E_i$ is surjective.
\end{enumerate}
Let $K_i = \ker a_i$ and note that these sheaves are all flat over $S$. Hence, restricting to the fibres of $s\in S$ we have an exact sequence
\[
0 \longrightarrow (K_i)_s \longrightarrow (A_i)_s \overset{a_i}{\longrightarrow} (E_i)_s \longrightarrow 0.
\]
We have $pr_2$ is smooth of relative dimension $1$ so any module on the fibres has homological dimension at most $1$. Thus showing that $(K_i)_s$ is locally free and hence $K_i$ is locally free. A finite complex of locally free sheaves quasi-isomorphic to $E^\bullet$ can be constructed from the total complex associated to the double complex of these resolutions. \\

By  \cite[Lemma 5]{FantechiPand} we can show $RF_* [\sigma]$ is torsion on $X\times S$ by showing is when $S$ is a point. Define $Y\subset C$ to be the locus where $f$ is not \'{e}tale and $Z= f(Y)\subset C$. Note that $Z$ is a finite collection of points in $X$. Define $\widetilde{C} = C\setminus Y$  with inclusion $j:\widetilde{C} \rightarrow C$ and $\widetilde{X} = X\setminus Z$  with inclusion $i:\widetilde{X} \rightarrow X$. We also have that $Y = \Supp(\mathrm{ker}\,\sigma)\cup \Supp(\mathrm{coker}\,\sigma)$ so $[j^*\sigma]$ is exact. Letting $\widetilde{f}=f|_{\widetilde{C}}$ be the restriction we have $i^* Rf_* [\sigma] = R\widetilde{f}_* [j^*\sigma]$. Hence, $i^* Rf_* [\sigma]$ is exact also, showing that the cohomology of $Rf_*[\sigma]$ is supported on points.
\end{proof}

\begin{lemma}\label{existence_of_regular_sections}
Let $S=\Spec\,A$ be Noetherian and let $E$ be a line bundle on $C$. There exists a $\rho$-relative line bundle $M$ on $C$ such that $H^0(C,M)$ and $H^0(C,M\otimes E)$ contain sections which define injective morphisms. 
\end{lemma}
\begin{proof}
Let $G$ be the bundle $\omega_{\overline{\rho}} \otimes \overline{f}^* \O_X(3)$ which is an ample $\overline{\rho}$-relative line bundle on $\overline{C}$. Let $n\in\mathbb{N}$ be large enough that both $G^N$ and $\overline{E}\otimes G^N$ are generated by global sections. We claim that $M := \gamma^* G^N$ has the desired properties. To see this note that $C$ is quasi-compact and so has a finite number of associated points. A standard argument then shows that the subspaces of $H^0(\overline{C},G^N)$ and $H^0(\overline{C},\overline{E}\otimes G^N)$ which are not injective morphisms will then have strictly lower dimension. Then consider the isomorphism
\[
\gamma_*: H^0(C, E\otimes M) \overset{\sim}{\longrightarrow} H^0(\overline{C}, \overline{E}\otimes G^N)
\]
and consider the pre-image $s$ of a regular section $\overline{s}\in H^0(\overline{C}, \overline{E}\otimes G^N)$. $\gamma_*$ is an exact functor and $\gamma_* K = 0$ if and only if $K=0$. Hence we have that $s$ is injective if and only if $\overline{s}$ is.
\end{proof}

\begin{lemma}\label{difference_of_divisors_is_zero}
Let $\widetilde{M}$ be a relative line bundle on $C$ and an injective morphism $s: \O_C \rightarrow  \widetilde{M}$. Let $D$ be the divisor on $C$ defined by $s^\vee$. Then
\[
\mathtt{D}:= \div\Big(RF_*\big[\O_D\otimes\O_C \overset{\mathrm{id}\otimes \sigma^k}{\longrightarrow } \O_D\otimes L^k\big]\Big) = 0.
\]
\end{lemma}

\begin{proof}
We first show that $\mathtt{D}$ is an effective divisor on $X\times S$ by considering the case where $S=\Spec A$. Note that the map forgetting the stack structure $\gamma:C\rightarrow \overline{C}$ has  the property that $\gamma_*$ is left exact. Also, $\O_{\overline{D}}$ is supported in relative dimension $0$, so $\mathtt{D}$ is given by
\[
\mathtt{D}= \div\Big(\overline{F}_*\big[\O_{\overline{D}} \,\longrightarrow  \gamma_*(\O_D\otimes L^k) \big]\Big).
\]

We have that $i:D\rightarrow C$ is a relative effective divisor on $C$ with coarse space $j:\overline{D}\rightarrow \overline{C}$. The natural map $\phi: \overline{D}\rightarrow S$ is quasi-finite and proper so it is also finite. Thus $\overline{D}$ is affine which shows that $(\gamma|_D)_*(\O_D\otimes L^k)$ is generated by global sections and so has sections which give injective morphism. \\

Let $\Phi$ be a section of $(\gamma|_D)_*(\O_D\otimes L^k)$ giving rise to an injective morphism. Then since  the divisor construction is additive on distinguished triangles we have:
\[
\mathtt{D}= \div\Big(\overline{F}_* j_*[\Phi ]\Big)
\]
Also,  $\Phi$ is regular so it is injective and we have $\mathtt{D}= \div( \overline{F}_* j_*\mathrm{coker}\,\Phi )$. Hence showing that it is a relative effective Cartier divisor. \\

The degree of a relative effective Cartier divisor for a smooth morphism is locally constant. Hence we can compute the degree at geometric points. We see that the degree of $(\mathtt{D})_z$ is zero for geometric points $z\in S$. 
\end{proof}

\begin{corollary}\label{divisor_tensor_product}
There is an equality of divisors
\[
\div\big(RF_*[\O_C \overset{\sigma^k}{\rightarrow} L^k]\big) = \div\big(RF_*(L \otimes [\O_C \overset{\sigma^k}{\rightarrow} L^k])\big)
\]
\end{corollary}
\begin{proof}
We have that $\div$ is additive on exact sequences. So, to show that two sequences give the same divisor, it will suffice to show that the cone of a morphism between the two complexes is the zero divisor.\\

We have two distinguished triangles coming from injective sections of $M$ and $M\otimes L^{-1}$, where $M$ is the line bundle from lemma \ref{existence_of_regular_sections}:
\[
[\sigma^k] \overset{s_1}{\longrightarrow} [\sigma^k]\otimes M \longrightarrow \mathrm{Cone}(s_1) \longrightarrow [\sigma^k] [1]
\]
\[
[\sigma^k]\otimes L \overset{s_2}{\longrightarrow} [\sigma^k]\otimes M \longrightarrow \mathrm{Cone}(s_2) \longrightarrow [\sigma^k] [1]
\]
We saw in the lemma \ref{difference_of_divisors_is_zero} that $\div(\mathrm{Cone}(s_1)) = \div(\mathrm{Cone}(s_2)) =0$ which shows that $[\sigma^k] $ and $[\sigma^k]\otimes L $ have the same divisor. 

\end{proof}

\begin{lemma}\label{distinguished_triangle_cone_composition}
Let $E^\bullet \overset{a}{\rightarrow} G^\bullet \overset{b}{\rightarrow} H^\bullet$ be morphisms in the derived category. Then there is a distinguished triangle:
\[
\mathrm{cone}(a) \longrightarrow \mathrm{cone}(b\circ a) \longrightarrow \mathrm{cone}(b) \longrightarrow \mathrm{cone}(a)[1]
\]
\end{lemma}
\begin{proof}
The result follows immediately from the following commuting diagram with distinguished triangles for rows and columns:
\[
\xymatrix@R=1.75em{
E^\bullet \ar[r]^{\mathrm{id}}\ar[d]^{a} & E^\bullet \ar[r]\ar[d]^{b\circ a}& 0 \ar[r] \ar[d]& E^\bullet[1] \ar[d]^{a[1] }\\
G^\bullet \ar[r]^{b}\ar[d] & H^\bullet \ar[r]\ar[d]& \mathrm{cone}(b) \ar[r] \ar[d]& G^\bullet[1] \ar[d]\\
\mathrm{cone}(a) \ar[r] \ar[d]& \mathrm{cone}(b\circ a) \ar[r]\ar[d]& \mathrm{cone}(b) \ar[r] \ar[d]& \mathrm{cone}(a)[1]\ar[d] \\
E^\bullet[1] \ar[r]^{\mathrm{id}[1] } & E^\bullet[1]  \ar[r]& 0 \ar[r] & E^\bullet[2] 
}
\]
\end{proof}

\begin{corollary} \label{distinguished_triangle_composition}
Let $a:E\rightarrow G$ and $b: G \rightarrow H$ be morphisms of coherent sheaves. Then there is a distinguished triangle:
\[
\big[\,E\overset{a}{\longrightarrow} G\,\big]
\longrightarrow
\big[\,E\overset{b\circ a}{\longrightarrow} C\,\big]
\longrightarrow
\big[\,G\overset{b}{\longrightarrow} H\,\big] 
\longrightarrow
\big[\,E\overset{a}{\longrightarrow} G\,\big] [1]
\]
where $[a]$, $[b]$ and $[b\circ a]$ are considered to be in degree $[-1,0]$.
\end{corollary}
\begin{proof}
The proof is immediate from lemma \ref{distinguished_triangle_cone_composition}.
\end{proof}

\begin{corollary}\label{divisor_power}
We have the equality of divisors on $X\times S$:
\[
\div\Big( RF_*[\O_C\overset{\sigma^r}{\longrightarrow} L^r]\Big) = 
r\cdot \div\Big( RF_*[\O_C\overset{\sigma}{\longrightarrow} L]\Big) 
\]
\end{corollary}
\begin{proof}
From corollary \ref{distinguished_triangle_composition} we have the distinguished triangle:
\[
\big[\,\O_C \overset{\sigma^n}{\longrightarrow} L^n \,\big]
\longrightarrow
\big[\,\O_C\overset{\sigma^{n+1}}{\longrightarrow} L^n\,\big]
\longrightarrow
\big[\,L\overset{\sigma}{\longrightarrow} L^{2}\,\big] 
\longrightarrow
\big[\,\O_C \overset{\sigma^n}{\longrightarrow} L^n \,\big][1].
\]
After applying corollary \ref{divisor_tensor_product} this shows that $\div\big( RF_*[\sigma^{n+1}]\big) = \div\big( RF_*[\sigma^n]\big) +\div\big( RF_*[\sigma]\big)$. The result follows from the induction hypothesis.
\end{proof}

\begin{corollary}
(Theorem \ref{br_theorem}) The divisor $\div( RF_*[\sigma])$ is relative effective and the associated morphism $b_\xi:S\rightarrow \Sym^m(X)$ defines a morphism of stacks:
\[\begin{array}{cccc}
\mathtt{br}:& \M^{[1/r]} &\longrightarrow &\Sym^m(X)\\
&\xi &\longmapsto &b_\xi.
\end{array}
\]
which satisfies the following commutative diagram:
\[
\xymatrix@R=1.75em{
 \M^{[{1/r}]} \ar[r]^{\mathtt{br}\hspace{0.5em}}\ar[d]_{\chi} &\Sym^{m} X \ar[d]^{\Delta}\\
 \M \ar[r]^{br\hspace{1em}} & \Sym^{rm} X  
}
\]
\end{corollary}
\begin{proof}
We have a natural quasi-isomorphism $[\sigma^r] \overset{\sim}{\rightarrow} [\delta]$. It is shown in \cite[3.2]{FantechiPand} that $\div( RF_*[\delta])$ is a relative effective divisor of degree $r m$. Hence, corollary \ref{divisor_power} shows that $\div( RF_*[\sigma])$ is relative effective as well and is of degree $m$. Corollary \ref{divisor_power} also shows that the given diagram is commutative. 
\end{proof}

\subsection{Special Loci of the Moduli Points}

In this subsection we will prove theorem \ref{br_theorem} part \ref{br_theorem_moduli_points} by considering the case when $S = \Spec\,\C$ and examining the ramification properties induced by the $r$th root condition.\\

Following \cite{Vakil_Enum, GraVakil} we will call a \textit{special loci} a connected component where the map $f:C\rightarrow X$ is not \'{e}tale. Then each special locus is one of:
\begin{enumerate}
\item A smooth point of $C$ where $f$ is locally of the form $z\mapsto z^{a+1}$ with $a\in\mathbb{N}$.
\item A node of $C$ such that on each brach $f$ is locally of the form $z\mapsto z^{a_i}$ with $a_i\in\mathbb{N}$.
\item A genus $g$ component $B$ of $C$ where $f|_{B}$ is constant and on the branches of $C$ meeting $B$ the map  $f$ is locally of the form $z\mapsto z^{a_i}$ with $a_i\in\mathbb{N}$.
\end{enumerate}
We can also define a \textit{ramification order} to each type of locus by:
\begin{enumerate}
\item $a$.  \hspace{1cm} (2) $a_1+a_2$. \hspace{1cm} (3) $2g_B-2 + \sum(a_i+1)$.\\
\end{enumerate}

\begin{remark} \label{relative_special_loci}
We use a slightly different definition for stable maps relative to a point $x\in X$. Let $(h: C\rightarrow T, p\rightarrow X)$ be over $S= \Spec\,\C$ in $\Mbar_g^{_{[1/r]}}(X,\mu)$ with $f=p\circ h$. Then a special locus of $f$ will be a connected component where the map $h:C\rightarrow T$ is not \'{e}tale and not in the pre-image of a node of $x$. Everything else is the same. This agrees with lemma \ref{lemma_ramification_morphism_degerated_target} which shows that $\delta$ will be an isomorphism at pre-images of nodes of $T$.
\end{remark}

We will show that the existence of an $r$th root of $\delta$ is equivalent to each of these special loci having ramification order divisible by $r$.\\

Suppose we have $\xi \in \M^{1/r}$ over $S = \Spec\,\C$. Then locally on the coarse space $\overline{C}$ for each of the types of special loci $\overline{\delta}:\O_{\overline{C}} \rightarrow \overline{\calR}_\xi$ is of the form:

\begin{enumerate}
\item $\C[x] \rightarrow \frac{1}{x^a} \C[x]$ given by $a \mapsto a\frac{x^{a}}{x^{a}}$.  
\item $\C[x,y]/(xy) \rightarrow \frac{1}{x^{a_1} - x^{a_2}} \C[x,y]/(xy)$ given by $a \mapsto a\frac{x^{a_1} - x^{a_2}}{x^{a_1} - x^{a_2}}$.
\item At each node $\C[x,y]/(xy) \rightarrow \frac{1}{x^{a_i} - x^{b_i}} \C[x,y]/(xy)$ given by $a \mapsto a\frac{x^{a_i}}{x^{a_i} - x^{b_i}}$.\\
\end{enumerate}

\begin{re} \label{smooth_nodal_rth_root}
The $r$th root condition $\sigma^r = e(\delta)$ forces there to be local roots for special loci of types $1$ and $2$. This forces the divisibility of the ramification order:
\begin{enumerate}
\item[]\textit{For type $1$:} Locally we must have $\sigma$ being of the form $\C[x] \rightarrow \frac{1}{x^{a/r}} \C[x]$ and thus $r$ divides $a$.  
\item[]\textit{For type $2$:} Pulling back from the coarse space via $\gamma$ we see that $\delta$ is of the form $\C[u,v]/(uv) \rightarrow \frac{1}{u^{a_1r} - v^{a_2 r}} \C[u,v]/(uv)$. Then taking the $r$th root we see that $\sigma$ is of the form $\C[u,v]/(uv) \rightarrow \frac{\zeta_r^k}{u^{a_1} - \zeta_r v^{a_2}} \C[u,v]/(uv)$ for some $k \in \mathbb{Z}/r$. However, there are multiplicities $e_1$ and $e_2$ of $L$ at the node with $e_1+e_2 = r$ or $e_1=e_2=0$. Also we have $a_i = e_i + n_ir$. Hence, $r$ divides $a_1+a_2$. 
\end{enumerate}

\end{re}

\begin{re}\label{contracted_component_ramification_equivalence}
We now consider special loci of type $3$. Suppose there is a genus $g$ sub-curve $B$ of $C$ where $f|_{B}$ is constant and on the branches of $C$ meeting $B$, the map  $f$ is locally of the form $z\mapsto z^{a_i}$ with $a_i\in\mathbb{N}$. \\
 
Let $A= C\setminus B$ and  $\alpha: A \sqcup B \rightarrow C$ be the partial normalisation of $C$ separating the contracted component $B$ from $A$. Also, let $p_i$ be the pre-images of the nodes on $A$ and $q_i$ the pre-images on $B$. Finally, let $a_i$ and $b_i$ be the multiplicities of $L$ corresponding to the branches on the nodes on A and B respectively. \\

Now $e$ restricts to an isomorphism $e_B : (L_B)^r \overset{\sim}{\longrightarrow} (\calR_\xi)_B\cong \omega_B(\sum q_i)$. We have a map $\mathsf{g}: B\rightarrow \mathsf{B}$ which forgets the orbifold structure at the points $q_i$. Pushing forward via $\mathsf{g}$ we have the following isomorphism coming from lemma \ref{rth_root_pushforward_divisor_lemma}:
\[
e_{\mathsf{B}} : (L_\mathsf{B})^r \overset{\sim}{\longrightarrow}  \omega_{\mathsf{B}}(\sum q_i-\sum b_i q_i)
\]
Hence, $\omega_{\mathsf{B}}(\sum q_i-\sum b_i q_i)$ must have degree divisible by $r$. Then $r$ divides $2g-2 + \sum(1-bi) $ and also divides $2g-2 + \sum(a_i +1) $. \\
\end{re}

\begin{remark}
To consider the relative case in \ref{contracted_component_ramification_equivalence} we must replace $f:C\rightarrow X$ by $h: C\rightarrow T$. Everything else remains the same.\\
\end{remark}

\begin{re}
To finish the proof of the theorem we to show that such an $r$th root can be constructed if the ramification loci are of the desired form. First we observe that there is a tensor product decomposition of $\delta$:
\[
\delta = \delta_{\mathrm{sm}}\otimes \delta_{\mathrm{n}} \otimes \delta_{\mathrm{cn}}
\]
where $\delta_{\mathrm{sm}}:\O_C\rightarrow R_{\mathrm{sm}}$ and $\delta_{\mathrm{n}}:\O_C\rightarrow R_{\mathrm{n}}$ define the divisors of $\delta$ supported on the smooth locus of $C$ and nodes of $C$ not meeting contracted components. Then $\delta_{\mathrm{cn}}:\O_C\rightarrow R_{\mathrm{cn}}$ is the unique section such that the above decomposition holds. \\

After reversing the reasoning of \ref{smooth_nodal_rth_root} we have $r$th roots $\sigma_{\mathrm{sm}}:\O_C\rightarrow L_{\mathrm{sm}}$ and $\sigma_{\mathrm{n}}:\O_C\rightarrow L_{\mathrm{n}}$ of $\delta_{\mathrm{sm}}$ and $\delta_{\mathrm{n}}$ respectively. For the contracted components $\delta_{\mathrm{cn}}$ the line bundle $R_{\mathrm{cn}}$ will locally of the form  $\frac{1}{u^{a_i r} - v^{b_i r}} \C[u,v]/(uv)$ at the connecting nodes and $\delta_{\mathrm{cn}}$ will be of the form  $1 \mapsto \frac{ u^{a_ir}}{u^{a_ir} - v^{b_i r}}$. Then let $L_{\mathrm{cn}}$ be an $r$th root of $R_{\mathrm{cn}}$ which is locally of the form  $\frac{1}{u^{a_i} - \zeta_r v^{b_i }} \C[u,v]/(uv)$ at the connecting nodes and $\sigma_{\mathrm{cn}}$ will be of the form  $1 \mapsto \frac{u^{a_i}}{u^{a_i} - v^{b_i }}$ and identical to $\delta_{\mathrm{cn}}$ elsewhere.\\

Hence we have proved theorem \ref{br_theorem} part \ref{br_theorem_moduli_points}.
\end{re}

\section{Cotangent Complex of \texorpdfstring{$\Mbar^{_{1/r}}_g(X, d)$}{Mg(X,d)}}\label{pot_section}

\textbf{Section \ref*{pot_section} Notation:} Recall the \hyperlink{notation_main}{notation convention}. The following diagram shows the relationships between the relevant spaces. It is commutative and many of the squares are cartesian. 
\begin{align}
\begin{tikzpicture}[baseline=(current  bounding  box.center), node distance=3.5cm, auto,
  f->/.style={->,preaction={draw=white, -,line width=3pt}},
  d/.style={double distance=1pt},
  fd/.style={double distance=1pt,preaction={draw=white, -,line width=3pt}}]
  \node (11) {$\mathsf{Tot}\,\bm{\pi}_*\L^r$};
  \node [right of=11] (12) {$\calC_{\mathsf{Tot}\,\bm{\pi}_*\L^r}$};
  \node [right of=12] (13) {$\mathsf{Tot}\,\L^r$};
  \node [below of=11, node distance=2.4cm] (21) {$\M^{[r]}$};
  \node [right of=21] (22) {$\calC^{[r]}$};
  \node [right of=22] (23) {$\calC^{[r]}$};
  \node (11b) [right of=11, above of=11, node distance=1.05cm] {$\mathsf{Tot}\,\pi_*\L$};
  \node [right of=11b] (12b) {$\calC_{\mathsf{Tot}\,\bm{\pi}_*\L}$};
  \node [right of=12b] (13b) {$\mathsf{Tot}\,\L$};
  \node [below of=11b, node distance=2.4cm] (21b) {$\M^{[r]}$};
  \node [right of=21b] (22b) {$\calC^{[r]}$};
  \node [right of=22b] (23b) {$\calC^{[r]}$};
  \node [above of =11, node distance=2.4cm] (01) {$\M^{[r]}$};
  \node [right of=01] (02) {$\calC^{[r]}$};  
  \node [right of=01, above of=01, node distance=1.05cm] (01b) {$\M^{[\frac{1}{r}]}$};
  \node [right of=01b] (02b) {$\calC^{[\frac{1}{r}]}$};
    \draw[->]  (02b) -> (01b) node[pos=0.7, above]{$\bm{\rho}$};
    \draw[->]  (12b) -> (11b) node[pos=0.7, above]{${\bm{\psi}}$}; \draw[->]  (12b) -> (13b) node[pos=0.5]{$\mathfrak{e}$};  
  \draw[->]  (22b) -> (21b) node[pos=0.7, above]{$\bm{\pi}$}; \draw[d]  (22b) -> (23b) ; 
 \draw[->] (11b) -> (21b) node[pos=0.7]{$\bm{\beta}$};  \draw[->] (12b) -> (22b) node[pos=0.7]{$\bm{\widehat{\beta}}$}; \draw[->] (13b) -> (23b) node[pos=0.7]{$\bm{\check{\beta}}$}; 
  \draw[->]  (01b) -> (11b) node[pos=0.7]{$\bm{i}$}; \draw[->]  (02b) -> (12b) node[pos=0.3]{$\bm{j}$};
  \draw[f->]  (02) -> (01) node[pos=0.4, above]{$\bm{\pi}$};
  \draw[f->]  (12) -> (11) node[pos=0.4, above]{${\bm{\varphi}}$}; \draw[f->]  (12) -> (13) node[pos=0.5]{$\mathfrak{e}'$};  
  \draw[f->]  (22) -> (21) node[pos=0.4, above]{$\bm{\pi}$}; \draw[fd]  (22) -> (23); 
  \draw[f->] (11) -> (21) node[pos=0.25, align=right, left]{$\bm{\alpha}$};  \draw[f->] (12) -> (22) node[pos=0.25]{$\bm{\widehat{\alpha}}$}; \draw[f->] (13) -> (23) node[pos=0.25]{$\bm{\check{\alpha}}$};
   \draw[f->]  (01) -> (11) node[pos=0.25, align=right, left]{$\bm{i'}$}; \draw[f->]  (02) -> (12) node[pos=0.25, align=right, left]{$\bm{j'}$};
   \draw[->]  (01b) -> (01) node[pos=0.8, align=left, above]{$\bm{\nu}$}; \draw[->]  (02b) -> (02) node[pos=0.7, align=left, above]{$\bm{\widehat{\nu}}$};
  \draw[->]  (11b) -> (11) node[pos=0.7, align=left, above]{$\bm{\tau}$}; \draw[->] (12b) -> (12) node[pos=0.7, align=left, above]{$\bm{\widehat{\tau}}$}; \draw[->] (13b) -> (13) node[pos=0.7, align=left, above]{$\bm{\check{\tau}}$}; 
  \draw[d] (21b) -> (21) ; \draw[d] (22b) -> (22) ;  \draw[d] (23b) -> (23) ; 
   \draw[->]  (02b) -> (13b) node[pos=0.45]{$\mathfrak{f}$};   \draw[f->]  (02) -> (13) node[pos=0.45]{$\mathfrak{f}'$};
\end{tikzpicture} 
\label{main_diagram}
\end{align}

Here $\bm{\pi}$ and $\bm{\rho}$ are the universal curves of their respective spaces. The maps $\bm{i}$ and $\bm{i'}$ are the natural inclusions defined by definition \ref{main_moduli_spaces_definitions}.\ref{full_space_part} and equation (\ref{inclusion_definition_delta}) respectively. The maps $\bm{\varphi}$ and $\bm{\psi}$ are the universal curves defined in \ref{totalspace_definition} with $\mathfrak{e}$ and $\mathfrak{e}'$ being the natural evaluation maps defined by equation (\ref{evaluation_map_definition}). The power maps $\bm{\tau}$ and $\bm{\check{\tau}}$ are defined in \ref{power_map_definition} and $\widehat{\bm{\tau}}$ is the pullback by $\bm{\varphi}$ of $\bm{\tau}$. The maps $\bm{\alpha}$, $\widehat{\bm{\alpha}}$, $\bm{\check{\alpha}}$, $\bm{\beta}$, $\widehat{\bm{\beta}}$ and $\bm{\check{\beta}}$ are the natural projection maps. The maps $\bm{j}$ and $\bm{j'}$ are pullbacks of $\bm{i}$ and $\bm{i'}$ by $\bm{\psi}$ and $\bm{\varphi}$ respectively. Lastly, we also define the maps $\mathfrak{f}:=\bm{i}\circ \mathfrak{e}$ and $\mathfrak{f}':=\bm{i'}\circ \mathfrak{e'}$. \\

We denote the expected number of special loci of order $r$ in the generic case by $m = \frac{1}{r} (2g-2-d(2g_X-2))$ in the case $\M=\Mbar_g(X,d)$ and  $m = \frac{1}{r}  (2g-2+l(\mu)+|\mu|(1-2g_X))$ in the case $\M=\Mbar_g(X,\mu)$.

\subsection{Perfect Relative Obstruction Theory}

Recall, that for a proper representable Gorenstein morphism $a: \calX \rightarrow \calY$ of relative dimension $n$ with relative dualising sheaf $\omega_{a}$ and any complexes $\calF^\bullet \in \mathrm{D}(\calX)$ and $\calG^\bullet \in \mathrm{D}(\calY)$ one has the following functorial isomorphism coming from Serre duality (see for example \cite[eq. C.12]{Fourier-Mukai_book}):
\begin{align}
\mathrm{Hom}_{\mathrm{D}(\calX)} \big(Ra_* \calF^\bullet, \calG^\bullet  \big) \overset{\sim}{\longrightarrow} \mathrm{Hom}_{\mathrm{D}(\calY)} \big( \calF^\bullet, a^*\calG^\bullet \otimes \omega_a[n]  \big). 
\end{align}
Hence one obtains the following natural morphism by looking at the pre-image of the identity: 
\begin{align}
Ra_* (a^* \calG^\bullet \otimes \omega_a )[n] \longrightarrow \calG^\bullet. \label{adjoint_map}
\end{align}

Now, consider the following sub-diagram of (\ref{main_diagram}) coming from the topmost horizontal square and the diagonal square: 
\[\xymatrix@=3em{ 
\M^{[1/r]}  \ar[d]_{\bm{\nu}} &\calC^{[1/r]}   \ar[l]_{\bm{\rho}} \ar[d]^{\bm{\widehat{\nu}}} \ar[r]^{\mathfrak{f}} & \mathsf{Tot}\,\L \ar[d]^{\bm{\check{\tau}}}\\
\M^{[r]} & \calC^{[r]}\ar[l]_{\bm{\pi}} \ar[r]^{\mathfrak{f}'} & \mathsf{Tot}\,\L^r
}
\]

There are two natural maps arising from this diagram:
\[
R\bm{\rho}_* (\bm{\rho}^* \bbL_{\bm{\nu}} \otimes \omega_{\bm{\rho}} )[1] \longrightarrow \bbL_{\bm{\nu}}
\hspace{1cm}\mbox{and}\hspace{1cm}
L\mathfrak{f}^*\bbL_{\bm{\check{\tau}}}  \longrightarrow \bbL_{\bm{\widehat{\nu}}} \cong  \bm{\rho}^* \bbL_{\bm{\nu}}.
\]
Combining these two we can we define the following morphism:
\[
\phi_{\bm{\nu}}: R\bm{\rho}_* (L\mathfrak{f}^*\bbL_{\bm{\check{\tau}}} \otimes \omega_{\bm{\rho}} )[1] \longrightarrow \bbL_{\bm{\nu}}.
\]
We will show in this subsection that this morphism is a perfect relative obstruction theory.\\

We will begin by examining a related morphism constructed in the same way. Specifically, we consider the  the following sub-diagram of (\ref{main_diagram}) coming from the middle horizontal squares:
\begin{align}
\begin{array}{c}
\xymatrix@=3em{ 
\mathsf{Tot}\,\bm{\pi}_*\L  \ar[d]_{\bm{\tau}} &\calC_{\mathsf{Tot}\,\bm{\pi}_*\L} \ar[l]_{{\bm{\psi}}} \ar[d]^{\bm{\widehat{\tau}}} \ar[r]^{\mathfrak{e}} &\mathsf{Tot}\,\L  \ar[d]^{\bm{\check{\tau}}} \\
\mathsf{Tot}\,\bm{\pi}_*\L^r & \calC_{\mathsf{Tot}\,\bm{\pi}_*\L^r}\ar[l]_{{\bm{\varphi}}} \ar[r]^{\mathfrak{e}'} & \mathsf{Tot}\,\L^r 
}
\end{array}\label{commuting_lemma_subdiagrams}
\end{align}
As before we have two natural maps 
\[
R{\bm{\psi}}_* ({\bm{\psi}}^* \bbL_{\bm{\tau}} \otimes \omega_{\bm{\psi}} )[1] \longrightarrow \bbL_{\bm{\tau}}
\hspace{1cm}\mbox{and}\hspace{1cm}
L\mathfrak{e}^*\bbL_{\bm{\check{\tau}}}  \longrightarrow \bbL_{\bm{\widehat{\tau}}} \cong  {\bm{\psi}}^* \bbL_{{\bm{\tau}}}.
\]
which combine to obtain the morphism:
\[
\phi_{{\bm{\tau}}}: R{\bm{\psi}}_* (L\mathfrak{e}^*\bbL_{\bm{\check{\tau}}} \otimes \omega_{\bm{\psi}} )[1] \longrightarrow \bbL_{\bm{\tau}}.
\]
The following lemma shows that $\phi_{{\bm{\tau}}}$ is a relative obstruction theory.

\begin{lemma} \label{relative_obstruction_theory_diagram_lemma}
There is a commuting diagram where the rows are distinguished triangles:
{\footnotesize
\[
\hspace{-1em}\xymatrix@C=1em{
L{\bm{\tau}}^* R{\bm{\varphi}}_* \big(\bm{\widehat{\alpha}}^*\L^{\SmNeg r} \otimes \omega_{\bm{\varphi}} \big)[1] \ar[r] \ar[d]^{L{\bm{\tau}}^*\phi_{\bm{\alpha}}}& R{\bm{\psi}}_* \big(\bm{\widehat{\beta}}^*\L^{\SmNeg 1} \otimes \omega_{\bm{\psi}} \big)[1]   \ar[r]\ar[d]^{\phi_{\bm{\beta}}} & R{\bm{\psi}}_* \big(L\mathfrak{e}^*\bbL_{\bm{\check{\tau}}} \otimes \omega_{\bm{\psi}} \big)[1] \ar[d]^{\phi_{{\bm{\tau}}}} \ar[r] &  L{\bm{\tau}}^* R{\bm{\varphi}}_* \big(\bm{\widehat{\alpha}}^*\L^{\SmNeg r} \otimes \omega_{\bm{\varphi}} \big)[2] \ar[d]^{\phi_{{\bm{\tau}}}[1]}\\
L{\bm{\tau}}^* \bbL_{\bm{\alpha}}  \ar[r] & \bbL_{\bm{\beta}}  \ar[r]& \bbL_{{\bm{\tau}}}  \ar[r] &L{\bm{\tau}}^* \bbL_{\bm{\alpha}} [1]
}
\]
}

such that $\phi_{\bm{\beta}}$ and $\phi_{\bm{\alpha}}$ are relative obstruction theories. Moreover, $\phi_{{\bm{\tau}}}$ is also a relative obstruction theory. 
\end{lemma}
\begin{proof}
Consider the leftmost square of (\ref{commuting_lemma_subdiagrams}) and note that it is cartesian. The distinguished triangle arising from the cotangent complex gives the following diagram where the rows are distinguished triangles: 
{\footnotesize
\begin{align}
\hspace{-1em}
\begin{array}{c}
\xymatrix@C=1em{
R{\bm{\psi}}_* ({\bm{\psi}}^*L{\bm{\tau}}^* \bbL_{\bm{\alpha}} \otimes \omega_{\bm{\psi}} )[1] \ar[r] \ar[d]^{}& R{\bm{\psi}}_* ( {\bm{\psi}}^*\bbL_{\bm{\beta}} \otimes \omega_{\bm{\psi}} )[1]   \ar[r]\ar[d]^{} & R{\bm{\psi}}_* ({\bm{\psi}}^*\bbL_{{\bm{\tau}}} \otimes \omega_{\bm{\psi}} )[1] \ar[d]^{} \ar[r] &  R{\bm{\psi}}_* ({\bm{\psi}}^*L{\bm{\tau}}^* \bbL_{\bm{\alpha}} \otimes \omega_{\bm{\psi}} )[2] \ar[d]^{}\\
L{\bm{\tau}}^* \bbL_{\bm{\alpha}}  \ar[r] & \bbL_{\bm{\beta}}  \ar[r]& \bbL_{{\bm{\tau}}}  \ar[r] &L{\bm{\tau}}^* \bbL_{\bm{\alpha}} [1]
}
\end{array} \label{Rel_POT_commute1}
\end{align}
}

We also have isomorphisms:
\[
R{\bm{\psi}}_* ({\bm{\psi}}^*L{\bm{\tau}}^* \bbL_{\bm{\alpha}} \otimes \omega_{\bm{\psi}} )
\cong
R{\bm{\psi}}_* L\widehat{{\bm{\tau}}}^* ({\bm{\varphi}}^*\bbL_{\bm{\alpha}} \otimes \omega_{\bm{\varphi}} )
\cong 
 L{\bm{\tau}}^* R{\bm{\varphi}}_* ({\bm{\varphi}}^*\bbL_{\bm{\alpha}} \otimes \omega_{\bm{\varphi}} )
\]
making the first column into the derived pullback of the canonical morphism from equation (\ref{adjoint_map}):
\[
R{\bm{\varphi}}_* ({\bm{\varphi}}^*\bbL_{\bm{\alpha}} \otimes \omega_{\bm{\varphi}} )[1] \longrightarrow \bbL_{\bm{\alpha}}.
\]

Now consider the rightmost square of (\ref{commuting_lemma_subdiagrams}) and note that  it has all morphisms over $\calC$. This gives the following commutative diagram with distinguished triangles as rows, noting that $L\bm{\widehat{\tau}}^* L\mathfrak{e}'^* \bbL_{\bm{\check{\alpha}}} \cong L\mathfrak{e}^* L\bm{\check{\tau}}^*  \bbL_{\bm{\check{\alpha}}}$:

\begin{align}
\begin{array}{c}
\xymatrix@=3em{
 L\bm{\widehat{\tau}}^* L\mathfrak{e}'^*\bbL_{\bm{\check{\alpha}}}\ar[r]\ar[d]& L\mathfrak{e}^*\bbL_{\bm{\check{\beta}}}  \ar[r]\ar[d]& L\mathfrak{e}^*\bbL_{\bm{\check{\tau}}}   \ar[d] \ar[r]&  \ar[d]  L\bm{\widehat{\tau}}^* L\mathfrak{e}'^*\bbL_{\bm{\check{\alpha}}}[1]\\
 L\bm{\widehat{\tau}}^* \bbL_{\bm{\widehat{\alpha}}} \ar[r] & \bbL_{\bm{\widehat{\beta}}} \ar[r]& \bbL_{\bm{\widehat{\tau}}}  \ar[r]&   L\bm{\widehat{\tau}}^* \bbL_{\bm{\widehat{\alpha}}}[1] \\
  L\bm{\widehat{\tau}}^*  {\bm{\varphi}}^*\bbL_{\bm{\alpha}} \ar[r] \ar[u]_\cong & {\bm{\psi}}^* \bbL_{\bm{\beta}} \ar[r]\ar[u]_\cong&{\bm{\psi}}^* \bbL_{{\bm{\tau}}} \ar[u]_\cong \ar[r]&    L\bm{\widehat{\tau}}^*  {\bm{\varphi}}^*\bbL_{\bm{\alpha}}[1] \ar[u]_\cong
}
\end{array} \label{Rel_POT_commute2}
\end{align}

Also, note that $ L\mathfrak{e}^* \bbL_{\bm{\check{\beta}}} \cong  L\mathfrak{e}^* (\bm{\check{\beta}}^*\L^\vee) \cong L(\bm{\check{\beta}} \circ \mathfrak{e})^*\L^\vee \cong \bm{\widehat{\beta}}^* \L^\vee$ and similarly, $ L\mathfrak{e}'^* \bbL_{\bm{\check{\alpha}}}  \cong \bm{\widehat{\alpha}}^* (\L^r)^\vee$. We now obtain the desired diagram by combining (\ref{Rel_POT_commute2}) with (\ref{Rel_POT_commute1}). We denote the appropriate morphisms by $\phi_{{\bm{\tau}}}$, $\phi_{\bm{\beta}}$ and $\phi_{\bm{\alpha}}$. It is shown in \cite[Prop. 2.5]{ChangLi} that $\phi_{\bm{\beta}}$ and $\phi_{\bm{\alpha}}$ are perfect relative obstruction theories.\\

To show that $\phi_{\bm{\tau}}$ is an obstruction theory it will suffice to show that $\calH^{\SmNeg 1}(\mathrm{cone}(\phi_{{\bm{\tau}}}))=\calH^0(\mathrm{cone}(\phi_{{\bm{\tau}}}))=0$. We have that $\bm{\beta}:\mathsf{Tot}\,\bm{\pi}_*\L^r\rightarrow \M$ is representable, so $\calH^{1}(\bbL_{\bm{\beta}}) = 0$ and $\calH^{i}(\phi_{\bm{\beta}}) = 0$ for all $i\geq -1$.  Also, $\mathrm{cone}(\phi_{\bm{\alpha}})$ is quasi-isomorphic to a flat complex $\calF^\bullet$ which is zero in all degrees greater than $-2$. Now by definition $L{\bm{\tau}}^* \mathrm{cone}(\phi_{\bm{\alpha}}) = {\bm{\tau}}^*\calF^\bullet$ also vanished in degrees greater than $-2$, making $\calH^{i}(L{\bm{\tau}}^*\phi_{\bm{\alpha}}) = 0$ for all $i\geq -1$. The result now  follows from taking the cohomology exact sequence of the distinguished triangle of the cones: 
\[\xymatrix@R=0.75em{
&\calH^{\SmNeg1}(\mathrm{cone}( \phi_{\bm{\beta}} ))\ar[r] &\calH^{\SmNeg1}(\mathrm{cone}(\phi_{{\bm{\tau}}})) \ar[r] &\calH^0(\mathrm{cone}(L{\bm{\tau}}^*\phi_{\bm{\alpha}})) \\
\ar[r] &\calH^0(\mathrm{cone}( \phi_{\bm{\beta}})) \ar[r]& \calH^0(\mathrm{cone}(\phi_{{\bm{\tau}}})) \ar[r]& \calH^1(\mathrm{cone}(L{\bm{\tau}}^*\phi_{\bm{\alpha}})).
}
\]
\end{proof}

\begin{lemma}
$\phi_{\bm{\nu}}$ is the composition of $L\bm{i}^* \phi_{\bm{\tau}}$ and the natural differential morphism $L\bm{i}^* \bbL_{\bm{\tau}} \rightarrow \bbL_{\bm{\nu}}$. In particular, $\phi_{\bm{\nu}}$ is a relative obstruction theory. 
\end{lemma}

\begin{proof}
There is a commuting diagram, noting that there is an isomorphism $L\bm{j}^* {\bm{\psi}}^* \bbL_{{\bm{\tau}}} \cong \bm{\rho}^* Lj^* \bbL_{{\bm{\tau}}}$: 
\[
\xymatrix{
L\bm{j}^* L\mathfrak{e}^* \bbL_{\bm{\check{\tau}}} \ar[r] \ar@{=}[d] & L\bm{j}^* \bbL_{\bm{\widehat{\tau}}}  \ar[d] & L\bm{j}^* {\bm{\psi}}^* \bbL_{{\bm{\tau}}} \ar[l]_{\cong} \ar[d]\\ 
L\mathfrak{f}^* \bbL_{\bm{\check{\tau}}} \ar[r] &\bbL_{\bm{\widehat{\nu}}} &\bm{\rho}^*\bbL_{{\bm{\nu}}} \ar[l]_{\cong}
}
\]
which gives the left square of the following diagram after applying the functor $R\bm{\rho}_*( \,\_\,  \otimes\omega_{\bm{\rho}})[1]$ and using the isomorphism of functors $R\bm{\rho}_*( L\bm{j}^*\,\_\,  \otimes\omega_{\bm{\rho}})[1] \cong L\bm{i}^* R{\bm{\psi}}_*( \, \_\, \otimes\omega_{\bm{\psi}})[1]$ :
\[
\xymatrix{
L\bm{i}^* R{\bm{\psi}}_*( L\mathfrak{e}^* \bbL_{\bm{\check{\tau}}} \otimes\omega_{\bm{\rho}})[1] \ar[r] \ar[d]_\cong & L\bm{i}^* R{\bm{\psi}}_*( {\bm{\psi}}^* \bbL_{{\bm{\tau}}} \otimes\omega_{\bm{\psi}})[1] \ar[d] \ar[r] & L\bm{i}^* \bbL_{{\bm{\tau}}}\ar[d] \\ 
R\bm{\rho}_*( L\mathfrak{f}^* \bbL_{\bm{\check{\tau}}} \otimes\omega_{\bm{\rho}})[1]\ar[r] &R\bm{\rho}_*(\bm{\rho}^*\bbL_{{\bm{\nu}}} \otimes\omega_{\bm{\rho}})[1] \ar[r] & \bbL_{\bm{\check{\tau}}}
}
\]
Now, $L\bm{i}^* \phi_{\bm{\tau}}$  is the composition of the top row and $\phi_{\bm{\nu}}$ is the composition of the bottom row. Hence, $\phi_{\bm{\nu}}$ is the composition of the desired morphisms. \\

The maps $\bm{i}$ and $\bm{i'}$ are immersions and ${\bm{\tau}}$ and ${\bm{\nu}}$ are representable, so $L\bm{i}^* \bbL_{\bm{\tau}} \rightarrow \bbL_{\bm{\nu}}$ is a relative obstruction theory (see for example \cite[\S 7]{Intrinsic}). We now consider the distinguished triangle of cones coming from composition of lemma \ref{distinguished_triangle_cone_composition}. The reasoning that $\phi_{\bm{\nu}}$ is a relative obstruction theory is now the same as for $\phi_{\bm{\tau}}$ in the previous lemma, lemma \ref{relative_obstruction_theory_diagram_lemma}.
\end{proof}

\begin{lemma}\label{global_sections_power_lemma}
The left derived pullback by $\mathfrak{f}$ of the map $\bm{\check{\tau}}^*\bbL_{\bm{\check{\alpha}}} \rightarrow \bbL_{\bm{\check{\beta}}}$ is the map:
\[
r\,\bm{j}^*\bm{\sigma}^{r-1} : \bm{j}^*\bm{\widehat{\beta}}^*\L^{\SmNeg r} \longrightarrow \bm{j}^*\bm{\widehat{\beta}}^*{\L}^{\SmNeg 1}
\] 
where $\sigma$ is the universal $r$th root. 
\end{lemma}
\begin{proof}
It will suffice to show this locally. The local situation is described by the diagram
\[
\xymatrix@R=2em{
\mathbb{A}^1\times U \ar[r]\ar[d]^{\mathtt{t}}& [\mathbb{A}^1\times U/G] \ar[r]\ar[d]& \mathsf{Tot}\,\L\ar[d]^{\bm{\check{\tau}}}\\
\mathbb{A}^1\times U \ar[r]\ar[d]^{\mathtt{a}}& [\mathbb{A}^1\times U/G] \ar[r]\ar[d]& \mathsf{Tot}\,\L^r\ar[d]^{\bm{\check{\alpha}}}\\
 U \ar[r]       & [U/G] \ar[r]        & \calC^{[r]}\\
}
\]
where $U=\Spec B$, $G$ is a finite group and $\mathtt{t}$ is defined by the morphism of $B$-algebras $B[z]\rightarrow B[w]$ with $z\mapsto w^r$. On $U$ we have that $\L$ is given by an equivariant line bundle $E$ defined by a local generator $\phi$. Then $\L^r$ corresponds to $E^r$ with $\phi^r$.\\

We have that $\bbL_{\bm{\check{\beta}}} \cong \Omega_{\bm{\check{\beta}}} \cong \bm{\check{\beta}}^*{\L}^{\SmNeg 1}$ and the last morphism is defined locally by $dw \mapsto \frac{1}{\phi}$. Similarly, $\bbL_{\bm{\check{\alpha}}} \cong \Omega_{\bm{\check{\alpha}}} \cong \bm{\check{\alpha}}^*{\L}^{\SmNeg r}$ is defined by $dz \mapsto \frac{1}{\phi^r}$ and the map $d\bm{\check{\tau}}$ is defined by:
\[
\begin{array}{cccc}
d{\mathtt{t}}:& B[w] \otimes A[z]dz& \longrightarrow & A[w] dw\\
& 1\otimes dz & \longmapsto &  rw^{r-1}dw.
\end{array}
\]

Let $\sigma_B \in B$ be the pullback of $\sigma$ to $U$. Locally on $U$ the map $\bm{j}$ is defined by the section $\sigma_B$ via the following map $B[w]\rightarrow B: w\mapsto \sigma_B $. Hence, pulling back via $\bm{j}$ we have that the map $\bm{j}^* d\bm{\check{\tau}}$ is locally defined by:
\[
\begin{array}{cccc}
\bm{j}^* d{\mathtt{t}}:& B \tfrac{1}{\phi^r}& \longrightarrow & B \tfrac{1}{\phi}\\
&  \tfrac{1}{\phi^r}& \longmapsto &   r(\sigma_B)^{r-1} \tfrac{1}{\phi}.
\end{array}
\]
\end{proof}

\begin{theorem}
The map $\phi_{{\bm{\nu}}}: R\bm{\rho}_* (L\mathfrak{f}^*\bbL_{\bm{\check{\tau}}} \otimes \omega_{\bm{\rho}} )[1] \longrightarrow \bbL_{\bm{\nu}}$ is a perfect relative obstruction theory perfect with relative virtual dimension $(1-r)m$. (Recall that $m$ is the dimension of $\M^{[r]}$.)
\end{theorem}

\begin{proof}
Let $\rho:C\rightarrow S$ be a family in $\M^{[\frac{1}{r}]}$. In the derived category we have the following isomorphisms
\[
L\mathfrak{f}^*\bbL_{\bm{\check{\tau}}} \otimes \omega_{\bm{\rho}}
 \cong  
L\mathfrak{f}^*\big([\check{\bm{\tau}}^*\bbL_{\bm{\check{\alpha}}} \longrightarrow \bbL_{\bm{\check{\beta}}}] \otimes \omega_{\bm{\rho}}\big)
 \cong 
\big[\,\widehat{\bm{\beta}}^* \L^{\SmNeg r} \otimes \omega_{\bm{\rho}}\longrightarrow \widehat{\bm{\beta}}^* \L^{\SmNeg 1}\otimes \omega_{\bm{\rho}} \,\big].
\]
Denote the restriction to $S$ of this complex by 
\begin{align}
E^\bullet = [E_{-1}\overset{\theta}{\longrightarrow} E_0] = [f^*\omega_{X} \otimes \O_C \overset{\mathrm{id}\otimes\sigma^{r-1}}{\longrightarrow} f^*\omega_{X} \otimes L^{r-1}] \label{theta_complex_definition}
\end{align}
with the last equality following from lemma \ref{global_sections_power_lemma}. Let $M$ be a line bundle on $C$ which is ample on the fibres of $\rho$. Then for sufficiently large $n$ we have for each $E_i$ the following properties:
\begin{enumerate}
\item\label{natural_map_surjective} $\rho^* \rho_*( E_i\otimes M^{n})\otimes M^{\SmNeg n} \longrightarrow E_i$ is surjective.
\item $R^1\rho_* E_i\otimes M^{n} = 0$.
\item\label{pull_push_vanishing property} For all $z\in S$ we have $H^0(C_z,  \rho^* \rho_*( E_i\otimes M^{n})\otimes M^{\SmNeg n}) =0$. 
\end{enumerate}

Denote the locally free sheaf $\rho^* \rho_*( E_0\otimes M^{n}) \otimes M^{-n}$ by $A_{E_0}$ and the associated map from property \ref{natural_map_surjective} above by $a$. Then using the fibre product for modules we have the following commuting diagram with exact rows
\begin{align}
\xymatrix{
0\ar[r] & \ker(a) \ar[r]\ar@{=}[d] & G\ar[r]\ar[d]^{\widetilde{\theta}} & {E_{\SmNeg1}} \ar[r] \ar[d]^\theta& 0\\
0\ar[r] & \ker(a) \ar[r] & A_{E_0} \ar[r]^{a} & {E_0} \ar[r] & 0
}\label{extension_construction_G}
\end{align}
where $G$ also fits into the exact sequence:
\begin{align}
\xymatrix{
0\ar[r] & G \ar[r] & {E_{\SmNeg1}} \oplus A_{E_{0}}\ar[r]^{\hspace{1em}\binom{\SmNeg \theta}{a}}& {E_{0}} \ar[r]& 0.
}\label{kernel_construction_G}
\end{align}
The diagram in (\ref{extension_construction_G}) shows that there is an isomorphism $[G\overset{\widetilde{\theta}}{\longrightarrow} A_{E_0}] \cong [E_{\SmNeg 1} \overset{\theta}{\longrightarrow} E_0]$ in the derived category.\\

The exact sequence in (\ref{kernel_construction_G}) shows that $G$ is locally free and hence the diagram in (\ref{extension_construction_G}) contains only flat modules. Hence for $z\in S$ we may restrict to the fibre $C_z$ and maintain exactness. Then using the snake lemma we have an isomorphism
\[
\ker \theta_z  \cong \ker \widetilde{\theta}_z . 
\]
We claim that $H^0(C_z, \ker \theta_z)=0$. To see this take $s \in H^0(C_z, \ker \theta_z)$ and note that $s \in H^0(C_z, E_{\SmNeg 1})$ and $s$ is in the kernel of $\theta_z$. From (\ref{theta_complex_definition}) we know that $(E_{\SmNeg 1})_z = f_z^*\omega_X$ and $\theta_z = \sigma_z^{r-1}$, so $\theta_z$ only vanishes where $f_z$ is constant. Hence, we let $B\subset C_z$ be the union of components contracted by $f_z$. Then we have $\Supp\,s \subset B$ and $(E_{\SmNeg 1})_z|_B =(f_z^*\omega_X)|_B \cong \O_B$ so we must have $s=0$. Hence, $H^0(C_z, \ker \theta_z)=H^0(C_z, \ker \widetilde{\theta}_z)=0$. \\

From property \ref{pull_push_vanishing property} of the definition of $A_{E_0}$ we have $H^0(C_z, (A_{E_0})_z) \cong 0$. So the following exact sequence shows that $ H^0(C_z, G_z) = 0$:
\[
0 \longrightarrow H^0(C_z, \ker\widetilde{\theta}_z)  \longrightarrow H^0(C_z,G_z)\longrightarrow H^0(C_z,  (A_{\widetilde{L}})_z)
\]
Hence, we have that $R^1\rho_* G$ is locally free and $R\rho_* G \cong [ R^1\rho_* G][-1]$. Moreover $R\rho_* [G\overset{\widetilde{\theta}}{\longrightarrow} A_{\widetilde{L}}] \cong [R^1\rho_* G \overset{R^1\rho_*\widetilde{\theta}}{\longrightarrow} R^1\rho_* A_{\widetilde{L}}][-1]$ is a complex of locally free sheaves concentrated in degree $[0,1]$. \\

The virtual dimension follows immediately from Riemann-Roch for twisted curves (see \cite[\S 7.2]{GW_of_DM}) applied to $E^\bullet$. 
\end{proof}

The space ${\M^{[r]}}$ has a natural perfect obstruction theory following from the construction of \cite{Behrend_GW, li2}. It will suffice to show that the space of $r$-stable maps $\M^r$ has a perfect obstruction theory since by lemma \ref{mod_div_ram_construction} the map $\M^{[r]}\rightarrow \M^r$ is \'{e}tale. Recall the construction of the perfect obstruction theory for the moduli space of stable maps $\M$ is defined via the relative to the forgetful morphism 
\[
\Mbar_g(X,d) \longrightarrow \frakM_g.
\]
It is pointed out in \cite[\S 2.8]{GraVakil} that in the case of relative stable maps the perfect obstruction theory can be constructed relative to the morphism 
\[
\Mbar_g(X,\mu) \longrightarrow \frakM_{g,l(\mu)}\times \calT_X
\]
where $\calT_X$ is the moduli space parameterising the degenerated targets. We have the following two cartesian squares where the bottom arrows are flat:
\[
\xymatrix{
 \Mbar^r_g(X,d) \ar[r]\ar[d]_{\bm{p}^{\mathrm{abs}}} & \Mbar_g(X,d) \ar[d]\\
\frakM^r_g \ar[r] &  \frakM_g
}
\hspace{1.5cm}
\xymatrix{
 \Mbar^r_g(X,\mu) \ar[r]\ar[d]_{\bm{p}^{\mathrm{rel}}} & \Mbar_g(X,\mu) \ar[d]\\
\frakM^r_{g,l(\mu)}\times \calT_X \ar[r] &  \frakM_{g,l(\mu)}\times \calT_X
}
\]
We let $\bm{p}:\M^r \rightarrow \calX$ be one of $\bm{p}^{\mathrm{abs}}$ or $\bm{p}^{\mathrm{rel}}$ maps depending on the choice of $\M^r$. Then we have a natural perfect relative obstruction for $\bm{p}$ by pulling back via $\M^r \rightarrow \M$:
\[
\phi_{\bm{p}}: E^\bullet_{\bm{p}} \longrightarrow \bbL_{\bm{p}}.
\]
Then a the perfect obstruction theory for $\M^r$ is given by the following cone construction:
\[
\xymatrix{
E_{\bm{p}}^\bullet[-1] \ar[r] \ar[d]^{\phi_{{\bm{p}}}[-1]} &  {\bm{p}}^* \bbL_{\calX}   \ar[r] \ar@{=}[d] &F_{\M^m}^\bullet  \ar[r] \ar[d]^\phi & E_{\bm{p}}^\bullet \ar[d]^{\phi_{{\bm{p}}}}  \\
\bbL_{{\bm{p}}}[-1] \ar[r] & {\bm{p}}^* \bbL_{\calX}  \ar[r] & \bbL_{\M^r}  \ar[r] & \bbL_{{\bm{\nu}}}
}
\]

\begin{corollary}
(Theorem \ref{pot_theorem}) If $g=0$ there is a perfect obstruction theory for $\M^{[1/r]}$ giving virtual dimension $m$. Moreover, since $\M^{[1/r]}\rightarrow \M^{1/r}$ is \'{e}tale in genus $0$, there is a perfect obstruction theory for $\M^{1/r}$. 
\end{corollary}
\begin{proof}
Let $E^\bullet = R\bm{\rho}_* (L\mathfrak{f}^*\bbL_{\bm{\check{\tau}}} \otimes \omega_{\bm{\rho}} )$. Then there the following is a commutative diagram with distinguished triangles for rows:
\[
\xymatrix{
E^\bullet[-1] \ar[r] \ar[d]^{\phi_{{\bm{\nu}}}[-1]} &  {\bm{\nu}}^* \bbL_{\M^{[r]}}   \ar[r] \ar@{=}[d] &F^\bullet  \ar[r] \ar[d]^\phi & E^\bullet \ar[d]^{\phi_{{\bm{\nu}}}}  \\
\bbL_{{\bm{\nu}}}[-1] \ar[r] & {\bm{\nu}}^* \bbL_{\M^{[r]}}  \ar[r] & \bbL_{\M^{[1/r]}}  \ar[r] & \bbL_{{\bm{\nu}}}
}
\]
Here $F^\bullet$ is defined via the cone construction. As before we have an exact sequence of cohomology of the cones:
\[\xymatrix@R=0.75em{
&\calH^{\SmNeg1}(\mathrm{cone}( \mathrm{id} ))\ar[r] & \calH^{\SmNeg1}(\mathrm{cone}(\phi)) \ar[r] &\calH^{\SmNeg1}(\mathrm{cone}(\phi_{{\bm{\nu}}})) \\
\ar[r] &\calH^{0}(\mathrm{cone}( \mathrm{id} ))\ar[r] & \calH^{0}(\mathrm{cone}(\phi)) \ar[r] &\calH^{0}(\mathrm{cone}(\phi_{{\bm{\nu}}})).
}
\]
Which shows that $\calH^{\SmNeg}(\mathrm{cone}(\phi)) = \calH^{0}(\mathrm{cone}(\phi)) = 0$.
\end{proof}

\end{document}